\documentclass[10pt]{article}

\usepackage{epsfig}
\usepackage{amssymb,amsmath,amsthm,url}
\usepackage{multirow}
\usepackage{booktabs}
\usepackage{natbib}
\usepackage{hyperref}
\usepackage{setspace}
\usepackage{dsfont}
\usepackage[labelformat=simple]{subcaption}

\usepackage{here}

\usepackage{amsfonts,times,eucal}
\usepackage{graphicx,ifthen}
\usepackage{subcaption}
\usepackage{wrapfig}
\usepackage{latexsym}
\usepackage{color}
\usepackage{mathrsfs}
\usepackage{bm}
\usepackage{bbm}
\usepackage{srctex}
\usepackage{enumitem}

\usepackage{mathtools}
\DeclarePairedDelimiter{\floor}{\lfloor}{\rfloor}

\theoremstyle{plain}
\newtheorem{theorem}{Theorem}[section]
\newtheorem{corollary}[theorem]{Corollary}
\newtheorem{lemma}[theorem]{Lemma}
\newtheorem{proposition}[theorem]{Proposition}

\newtheorem{condition}[theorem]{Condition}

\theoremstyle{definition}
\newtheorem{definition}[theorem]{Definition}

\theoremstyle{remark}

\numberwithin{equation}{section}

\newcommand{\N}{\mathbb{N}}
\newcommand{\R}{\mathbb{R}}

\newcommand{\Z}{\mathbb{Z}}
\newcommand{\I}{\mathds{1}}

\newcommand{\p}{\mathbbm{P}}

\newcommand{\cA}{\mathcal{A}}
\newcommand{\cB}{\mathcal{B}}

\newcommand{\cF}{\mathcal{F}}

\newcommand{\cN}{\mathcal{N}}

\newcommand{\cO}{\mathcal{O}}

\newcommand{\E}[1]{\mathbb{E}\left [ \, #1 \, \right ]}

\renewcommand{\epsilon}{\varepsilon}
\renewcommand{\phi}{\varphi}

\newcommand{\intd}[1]{\,\mathrm{d}#1}
\newcommand{\norm}[1]{\left\lVert #1 \right\rVert}
\newcommand{\scalar}[2]{\left\langle #1,#2 \right\rangle}
\newcommand{\1}[1]{\,\mathbbm{1}\! \left\{ #1 \right\} }

\newcommand{\darg}{\,\cdot\,}
\newcommand{\supp}{\text{supp}\,}

\newcommand{\dzn}[1]{d_{\infty}\left( #1 \right) }
\allowdisplaybreaks
\begin{document}

\title{Nonparametric Density Estimation for Spatial Data with Wavelets\footnote{This research was supported by the Fraunhofer ITWM, 67663 Kaiserslautern, Germany which is part of the Fraunhofer Gesellschaft zur F{\"o}rderung der angewandten Forschung e.V.}}
\author{Johannes T. N. Krebs\footnote{Department of Statistics, University of California, Davis, One Shields Avenue, 95616, CA, USA, email: \tt{jtkrebs@ucdavis.edu} }\; \footnote{Corresponding author} }

\date{\today}
\maketitle

\begin{abstract}
\setlength{\baselineskip}{1.8em}
Nonparametric density estimators are studied for $d$-dimensional, strongly spatial mixing data which is defined on a general $N$-dimensional lattice structure. We consider linear and nonlinear hard thresholded wavelet estimators which are derived from a $d$-dimensional multiresolution analysis. We give sufficient criteria for the consistency of these estimators and derive rates of convergence in $L^{p'}$ for $p'\in [1,\infty)$. For this reason, we study density functions which are elements of a $d$-dimensional Besov space $B^s_{p,q}(\R^d)$. We also verify the analytic correctness of our results in numerical simulations.\medskip\\
\noindent {\bf Keywords:} Besov spaces; Density estimation; Hard thresholding; Rate of convergence; Spatial lattice processes; Strong spatial mixing; Wavelets

\noindent {\bf MSC 2010:} Primary: 62G07; 62H11; 65T60; Secondary: 65C40; 60G60
\end{abstract}

This article considers methods of nonparametric density estimation for spatially dependent data with wavelets. There is an extensive literature on the density estimation problem for i.i.d.\ data or time series. Recently, inference techniques for spatial data have gained importance because of their increased relevance in modern applications such as image analysis, forestry, epidemiology or geophysics. See the monographs of \cite{cressie1993statistics} and \cite{guyon1995random} for a systematic introduction on spatial data and random fields.

So far when working with random fields, the kernel method has been a popular tool both in regression and density estimation, see, e.g., \cite{carbon1996kernel}, \cite{hallin2001density}, \cite{hallin2004local}, \cite{biau2003spatial} and \cite{carbon2007kernel}. In a more recent paper, \cite{dabo2013kernel} extend the kernel method to functional stationary random fields, they estimate the spatial density w.r.t.\ a reference measure. \cite{dabo2014kernel} propose a kernel method in spatial density estimation which also allows for spatial clustering. \cite{amiri2016nonparametric} study asymptotic properties of a recursive version of the Parzen-Rozenblatt estimator.

While the kernel method is efficient if the density has unbounded support, it often has disadvantages for densities with compact support because of the boundary bias. Furthermore, the kernel method requires the density to satisfy certain smoothness conditions. In situations where the density function does not meet these requirements the wavelet method is an alternative which often performs relatively well because it adapts automatically to the regularity of the curve to be estimated. Wavelet estimators assume that the underlying curve belongs to a function space with certain degrees of smoothness. The wavelet estimators do not depend on the smoothness parameters, nevertheless, they behave as if the true curve is known in advance and attain the optimal rates of convergence. This is in particular true for the hard thresholding estimator of \cite{donoho1996}. 

However, estimating the density of spatial data with wavelets has received little attention. Only the special case of time series has been thoroughly investigated: \cite{masry1997multivariate}, \cite{masry2000wavelet}, \cite{bouzebda2015multivariate} and \cite{bouzebda2017multivariate} study the wavelet method for density and regression estimators for multivariate and stationary time series. In a recent article \cite{LiWavelets} studies wavelet estimators for compactly supported one-dimensional Besov densities on stationary and strongly mixing random fields.

In the present article, we continue with these considerations for $d$-dimensional densities and study the linear and the hard thresholding estimator based on not necessarily isotropic wavelets. It is well-known that the hard thresholding estimator performs better than its linear analogue for certain densities in the one-dimensional setting. We will show a similar behavior for multivariate density functions.

The hard thresholding estimator has a linear basic component w.r.t.\ a coarse level $j_0$. Additionally, nonlinear details are added for higher levels $j_0\le j \le j_1$ if their contribution is significant in the statistical sense. This implies that this estimator can converge faster than the linear estimator in certain parameter settings.

The generalization to arbitrary dimensions is non-trivial, in particular because the definitions of the underlying Besov space $B^s_{p,q}(\R^d)$ have to be generalized to the $d$-dimensional case. For isotropic wavelets there already exist such generalizations, see for instance \cite{meyer1990ondelettes} or \cite{haroske2005wavelet}. However, as we also allow for density estimators with nonisotropic wavelets, we need a more general definition. This is one of the main differences to the existing work. Moreover, we allow for density functions on $\R^d$ which do not necessarily have compact support.

We assume that $Z = \{ Z(s): s\in \Z^N \}$ is a random field with equal marginal laws on $\R^d$ which admit a square integrable density $f$ w.r.t.\ to the $d$-dimensional Lebesgue measure $\lambda^d$. Then for an orthonormal basis $\{ b_u: u\in \N_+ \}$ of $L^2(\lambda^d)$ there is the representation $f = \sum_{u\in\N_+} \scalar{f}{b_u} b_u$, where $\scalar{\cdot}{\cdot}$ is the inner product on the function space $L^2(\lambda^d)$. Since $f$ is a density, we have the fundamental relationship between an observed sample $\{ Z(s):s\in I \}$ ($I\subseteq \Z^N$) and a coefficient $\scalar{f}{b_u}$ from this representation: $\scalar{f}{b_u} = \E{ b_u(Z(s)) } \approx |I|^{-1} \sum_{s\in I} b_u(Z(s))$.

It is well-known that replacing the true coefficient with the empirical approximation yields a consistent density estimate for an i.i.d.\ sample of one-dimensional data under certain conditions, see, e.g., \cite{devroye1985nonparametric} or \cite{hardle2012wavelets}. In the particular case of wavelets, \cite{kerkyacharian1992density} derive rates of convergence for the linear wavelet estimator.

In contrast to linear wavelet estimators nonlinear wavelet estimators are particularly useful if the density curve features high-frequency oscillations or shows an erratic behavior. Rates of convergence of the hard thresholded wavelet estimator are studied by \cite{hall1995} and \cite{donoho1996}. Since then the wavelet method for the density problem has been studied in various special settings: \cite{hall1998block}, \cite{cai1999adaptive} and \cite{chicken2005block} consider rates of convergence for wavelet block thresholding. \cite{gine2009uniform} give several uniform limit theorems for wavelet density estimators for a compactly supported density and i.i.d.\ sample data. \cite{xue2004approximation} study wavelet based density estimation under censorship. \cite{gine2014wavelet} investigate wavelet projection kernels in the density estimation problem. In this article, we continue the analysis for multivariate sample data which is spatially dependent.

This manuscript is organized as follows: we give the fundamental definitions and summarize the main facts of Besov spaces in $d$ dimensions in Section~\ref{Section_NotationDefinitions}. In Section~\ref{WaveletResults} we study in detail the wavelet density estimators. We give criteria which are sufficient for the consistency of the nonparametric estimators and establish rates of convergence. Section~\ref{ExamplesOfApplication} is devoted to numerical applications. We use an algorithm proposed by \cite{kaiser2012} for the simulation of the random field and estimate its marginal density with the linear and the hard thresholded wavelet estimator. Section~\ref{Proofs_of_the_main_theorems} contains the proofs of the results from Section~\ref{WaveletResults}. Appendix~\ref{Appendix_Exponential inequalities} contains useful inequalities for dependent sums. As the wavelet estimators are a priori not necessarily a density, we consider in Appendix~\ref{Appendix_QuestionOfNormalization} the question under which circumstances a normalized estimator is consistent.

\section{Notation and Definitions}\label{Section_NotationDefinitions}
This section is divided in four parts: firstly, we introduce the concepts for multidimensional wavelets. Secondly, we define the multidimensional Besov spaces. We explain the data generating process in the third step. Finally, we define the wavelet density estimator. 

In the following, we write $L^2(\lambda^d)$ for $L^2\left(\R^d, \cB(\R^d), \lambda^d \right)$, where $\lambda^d$ is the $d$-dimensional Lebesgue measure and we write $\norm{f}_{L^p(\lambda^d)} = (\int_{\R^d} |f|^p\intd{\lambda^d} )^{1/p}$ for the $L^p$-norm of a function $f$ on $\R^d$.

We begin with well-known results on wavelets in $d$ dimensions, see, e.g., the monograph of \cite{benedetto1993wavelets}.
 
\begin{definition}\label{MRA}
Let $\Gamma \subseteq \R^d$ be a lattice, this is a discrete subgroup given by	$( \Gamma, + ) =\left( \left\{ \sum_{i=1}^d a_i v_i: a_i \in \Z \right\}, +\right)$ for certain $v_i \in \R^d$ ($i=1,\ldots,d$). Furthermore, let $M \in \R^{d\times d}$ be a matrix which preserves the lattice $\Gamma$, i.e., $M\Gamma \subseteq \Gamma$ and which is strictly expanding, i.e., all eigenvalues $\zeta$ of $M$ satisfy $|\zeta| > 1$. Denote for such a matrix $M$ the absolute value of its determinant by $|M|$. A multiresolution analysis (MRA) of $L^2\left(\R^d, \cB(\R^d), \lambda^d \right)$, $d \in \N_+$, with a scaling function $\Phi:\R^d\rightarrow\R$ is an increasing sequence of subspaces of $L^2\left(\lambda^d \right)$ given by $\ldots \subseteq U_{-1} \subseteq U_0 \subseteq U_1 \subseteq \ldots $ such that the following four conditions are satisfied
\begin{enumerate}
\item (Denseness) $\bigcup_{j \in \Z} U_j$ is dense in $L^2\left(\lambda^d \right)$,
\vspace{-.5em}
\item (Separation) $\bigcap_{j \in \Z} U_j = \{0\}$,
\vspace{-.5em}
\item (Scaling) $ f\in U_j $ if and only if $f( M^{-j} \,\cdot\,) \in U_0$,
\vspace{-.5em}
\item (Orthonormality) $\{ \Phi(\,\cdot\, - \gamma): \gamma \in \Gamma \}$ is an orthonormal basis of $U_0$.
\end{enumerate}
\end{definition}

The relation between an MRA and an orthonormal basis of $L^2(\lambda^d)$ is summarized by

\begin{theorem}[\cite{strichartz1993wavelets}]\label{BenedettoTheorem}
Suppose $\Phi$ generates a multiresolution analysis and the $a_k(\gamma)$ satisfy for all $0\le j,k \le |M|-1$ and $\gamma\in\Gamma$ the equations
\begin{align*}
	\sum_{\gamma'\in\Gamma} a_j(\gamma')\, a_k(M\gamma + \gamma') = |M|\, \delta(j,k)\, \delta(\gamma,0) \quad\text{ and }\quad \sum_{\gamma\in\Gamma} a_0(\gamma) = |M|,
\end{align*}
where $\delta$ is the Kronecker delta. Furthermore, let the functions $\Psi_k$ be defined by $\sum_{\gamma\in\Gamma} a_k(\gamma)\, \Phi(M\,\cdot\, - \gamma)$ for $k=1,...,|M|-1$. Then the set of functions $\{ |M|^{j/2} \Psi_k( M^j\,\cdot\,-\gamma): j\in\Z, k=1,\ldots,|M|-1, \gamma\in\Gamma \}$ forms an orthonormal basis of $L^2(\lambda^d)$:
\begin{align*}
		&L^2(\lambda^d) = U_0 \oplus \left( \oplus_{j\in\N} W_j \right)= \oplus_{j\in\Z} W_j,\text{ where } W_j \coloneqq \langle\, |M|^{j/2} \Psi_k(M^j\,\cdot\, - \gamma): k=1,\ldots,|M|-1,\gamma\in\Gamma \, \rangle.
\end{align*}
\end{theorem}

We also call the scaling function $\Psi_0=\Phi$ the father wavelet. Moreover, we assume throughout the rest of this article that the MRA is given by compactly supported and bounded wavelets $\Psi_k$, $k=0,\ldots,|M|-1$ if not mentioned otherwise. Additionally, we assume that the lattice $\Gamma$ is $\Z^d$. One could also use different lattices which would have a finer grid than $\Z^d$, however, this would also result in more technical complexities and provide little additional insight. Note that the last assumption also implies the eigenvalues of the matrix $M$ to be integers.

W.l.o.g. the support of the wavelets $\Psi_k$ is in $[0,L]^d$ for some $L\in\N_+$, we write $\operatorname{supp} \Psi_k \subseteq [0,L]^d$. The mother wavelets satisfy the balancing condition $\int_{\R^d} \Psi_k \,\intd{\lambda^d} =0$ for $k=1,\ldots,|M|-1$.

One can derive a $d$-dimensional, isotropic MRA from a father wavelet $\phi$ and a mother wavelet $\psi$ which are defined on the real line: assume that $\phi$ and $\psi$ fulfill the scaling equations
\begin{align*}
		\phi \equiv \sqrt{2} \sum_{\gamma\in\Z}  h_\gamma\, \phi( 2 \darg - \gamma) \text{ and } \psi \equiv \sqrt{2} \sum_{\gamma\in\Z} g_\gamma\, \phi (2 \darg -\gamma),
\end{align*}
for real sequences $( h_\gamma: \gamma \in \Z)$ and $(g_\gamma: \gamma\in\Z)$. Let $\phi$ generate an MRA of $L^2(\lambda)$ with the corresponding spaces $U'_j$, $j\in \Z$. The $d$-dimensional wavelets are derived as follows: set $\Gamma \coloneqq \Z^d$ and define the diagonal matrix $M$ as $2 I$, where $I$ is the identity matrix. Denote the mother wavelets as pure tensors by $\Psi_k \coloneqq \xi_{k_1} \otimes \ldots \otimes \xi_{k_d}$ for $k \in \{0,1\}^d\setminus \{0\}$, where $\xi_0 \coloneqq \phi$ and $\xi_1 \coloneqq \psi$. The scaling function is $\Phi \coloneqq \Psi_0 \coloneqq \otimes_{i=1}^{d} \phi$. Then $\Phi$ and the linear spaces $U_j \coloneqq \otimes_{i=1}^d U'_j$ form an MRA of $L^2(\lambda^d)$ and the functions $\Psi_k$, $k\neq 0$, generate an orthonormal basis in that
\begin{align*}
		&L^2(\lambda^d) = U_0 \oplus \left( \oplus_{j\in\N} W_j \right) = \oplus_{j\in\Z} W_j, \text{ where } W_j = \left\langle\, |M|^{j/2} \Psi_k\left( M^j\darg - \gamma \right): \gamma \in \Z^d,\, k \in \{0,1\}^d \setminus \{0\} \, \right\rangle.
\end{align*}

Moreover, we need the characterization of the multivariate Besov space, see \cite{meyer1990ondelettes} or \cite{triebel1992theory}. For that reason we generalize the well-known multivariate notion of the Besov space for isotropic wavelets to nonisotropic wavelets.

\begin{definition}\label{DefBesovSpace}
Let $s > 0$, $p, q \in [1,\infty]$ and let $\{ \Psi_0,\ldots, \Psi_{|M|-1} \}$ be a wavelet basis. Set $	\Phi_{j,\gamma} \coloneqq \Psi_{0,j,\gamma} \coloneqq |M|^{j/2} \, \Phi ( M^j \,\cdot\, - \gamma )$	for the father wavelets and write $\Psi_{k,j,\gamma} \coloneqq |M|^{j/2} \, \Psi_k( M^{j}\,\cdot\, - \gamma)$ for the mother wavelets for $k=1,\ldots,|M|-1$, $j\in \Z$ and $\gamma\in\Z^d$. The Besov space $B^s_{p,q}(\R^d)$ is defined as
\begin{align*}
	B^s_{p,q}(\R^d) &\coloneqq \left\{ f: \R^d \rightarrow \R, \text{ there is a wavelet representation }  \right .  \\
	& \qquad \qquad \left. f = \sum_{\gamma\in\Z^d} \theta_{0, \gamma}\, \Phi_{0,\gamma} + \sum_{k=1}^{|M|-1} \sum_{j \ge 0} \sum_{\gamma\in\Z^d} \upsilon_{k,j,\gamma}\, \Psi_{k,j,\gamma} \text{ such that } \norm{f}_{B^s_{p,q}} < \infty \right\},
\end{align*}
where the Besov norm of $f$ (with the usual modification if $p=\infty$ or $q=\infty$) is
\begin{align}\begin{split}\label{GeneralizedBesovNorm}
		\norm{f}_{B^s_{p,q} } &\coloneqq \norm{\sum_{\gamma\in\Z^d} \theta_{0, \gamma}\, \Phi_{0,\gamma} }_{L^p(\lambda^d)} + \left( \sum_{k=1}^{|M|-1} \sum_{j \ge 0} |M|^{jsq} \norm{ \sum_{\gamma\in\Z^d} \upsilon_{k,j,\gamma} \, \Psi_{k,j,\gamma} }_{L^p(\lambda^d)}^q		\right)^{1/q}.
\end{split}
\end{align}
Furthermore, denote the $\ell^p$-sequence norm by $\norm{\,\cdot\,}_{\ell^p}$ and define the equivalent norms (see Lemma~\ref{BesovEquivalence})
\begin{align}\label{GeneralizedBesovNorm2}
		\norm{f}_{s,p,q} \coloneqq \norm{ \theta_{0, \cdot}  }_{\ell^p } + \left( \sum_{k=1}^{|M|-1} \sum_{j \ge 0} |M|^{j(s+1/2-1/p)q} \norm{  \upsilon_{k,j,\cdot} }_{\ell^p}^q		\right)^{1/q}.
\end{align}

In the following, $M$ is a diagonalizable matrix, $M = S^{-1} D S$, where $D$ is a diagonal matrix containing the eigenvalues of $M$. Denote the maximum of the absolute values of the eigenvalues by $\zeta_{max} \coloneqq \max\{ |\zeta_i| : i=1,\ldots, d \}$ and the corresponding minimum by $\zeta_{min} \coloneqq \min\{ |\zeta_i| : i=1,\ldots, d \}$.

Similar to the case in one dimension, we have the following relations between different Besov spaces for multivariate functions, cf. \cite{donoho1996}.
\begin{enumerate}
		\item If either $s'>s$ and $q=q'$ or if $s'=s$ and $q'\le q$, then $B^{s'}_{p,q'} \subseteq  B^s_{p,q}$. Moreover, if $p'\le p$ and $s'=s-p^{-1}+(p')^{-1}$, then $B^{s'}_{p',q} \subseteq B^s_{p,q}$.
		
	\item If $s'=s-p^{-1}>0$, then $B^s_{p,q} \subseteq B^s_{p,\infty} \subseteq B^{s'}_{\infty,\infty}$.
	
	\item  Furthermore, if a function is H{\"o}lder continuous with exponent $0<r\le 1$, we see in the following that this function belongs to the Besov space $B^s_{\infty,\infty}$, where the regularity parameter $s$ is given by $r \ln \zeta_{min} / (d \zeta_{max}) $. In particular, in the one-dimensional case $s$ equals $r$, otherwise it is strictly smaller.
	
\end{enumerate}

Moreover, a wavelet is $r$-regular if every derivative up to order $r\in\N_+$ is rapidly decreasing. In the one-dimensional case, this regularity ensures that the characterization of the Besov norms via the wavelet coefficients as in \eqref{GeneralizedBesovNorm} and \eqref{GeneralizedBesovNorm2} is equivalent to the characterization via the modulus of smoothness, compare \cite{Lemarie1986} and \cite{donoho1997universal}.
In the one-dimensional case and for $r$-regular wavelets, the Besov spaces also include the Sobolev spaces $H^s = B^s_{2,2}$. Similar considerations remain true (at least) in the special case of isotropic wavelets. For more details, we refer the reader to \cite{meyer1990ondelettes} and \cite{haroske2005wavelet}.
We do not consider such equivalent characterizations for a general matrix $M$ in the following but leave this issue up to further research.

For the density estimation problem, we define for $K \in \R_+$, $A\in \cB(\R^d)$ and $d\in\N_+$ subsets of $B^s_{p,q}$  as follows
\begin{align*}
			&F_{s,p,q}(K,A) \coloneqq \left\{f:\R^d \rightarrow\R_{\ge 0}, f \in B^s_{p,q}(\R^d), \int_{\R^d} f \intd{\lambda^d} = 1, \norm{f}_{s,p,q} \le K, \supp{f} \subseteq A 		\right\}.
\end{align*}
\end{definition}

If the wavelets $\Psi_k$ have compact support and if $s-1/p > 0$, then it is straightforward to show that finiteness of $f$ w.r.t.\ the Besov norm implies that the function is essentially bounded by $\norm{f}_{s,p,q}$ times a constant. In particular, if $f$ is a density such that $\norm{f}_{s,p,q} < \infty$ and $s>1/p$, then $f$ is square integrable.

In the statements below, the notation $|M|^{j} \simeq g(n)$ means that the integer $j$ is chosen as a function of $n$ such that $|M|^j \le g(n) < |M|^{j+1}$.

We denote the $p$-norm on $\R^N$ (resp. $\R^d$) by $\norm{\,\cdot\,}_{p}$ and the corresponding metric by $d_{p}$ for $p\in[1,\infty]$ with the extension $d_p(I,J) \coloneqq \inf\{	d_p(s,t), s\in I, t\in J	\}$ to subsets $I,J$ of $\R^N$ (resp. $\R^d$). Write $s \le t$ for $s,t \in \R^N$ if and only if for each $1\le i \le N$ the single coordinates satisfy $s_i \le t_i$. We denote the indicator function of a set $A$ by $\I\{A\}$. For $a\in\R$ we write $a^+ \coloneqq \max(a,0)$ for the positive and $a^- \coloneqq \max(-a,0)$ for the negative part. Furthermore, we write $e_N \coloneqq (1,\ldots,1)\in\Z^N$ for the vector whose elements are all equal to one. If $a,b\in \R^d$ are such that $a\le b$, then we denote the cube $\{x\in\R^d: a\le x \le b\}$ by $[a,b]$.

We call a function $h: \R^d \rightarrow \R$ radial if $h(x) = h(y)$ whenever $\norm{x}_2 = \norm{y}_2$. A radial function $h$ is non-increasing if $h(x) \le h(y)$ whenever $\norm{x}_2 \ge \norm{y}_2$.

In the next step, we describe the data generating process which is given by a $d$-dimensional random field $Z$. This random field is defined on an $N$-dimensional lattice structure, i.e., $Z = \{ Z(s): s \in \Z^N \}$ ($N\ge 1$). The random variables $Z(s)$ are identically distributed on $\R^d$ and their distribution admits a density $f$.

Denote for a subset $I$ the $\sigma$-algebra generated by the $Z(s)$ in $I$ by $\cF(I) = \sigma( Z(s): s\in I)$. The $\alpha$-mixing coefficient is introduced in \cite{rosenblatt1956central}; in the present context it is defined for $k\in\N$ as
\begin{align*}
	\alpha(k) \coloneqq \sup_{ \substack{I, J \subseteq \Z^N,\\ \dzn{I,J}\ge k }} \sup_{ \substack{ A \in \cF(I),\\ B \in \cF(J) } } \left| \p(A\cap B)-\p(A)\p(B)		\right|.
\end{align*}
We say that the random field is strongly spatial mixing if $\alpha(k) \rightarrow 0$ for $k \rightarrow \infty$. \cite{bradley2005basicMixing} gives an introduction to dependence measures for random variables. In the following, we require
\begin{condition}\label{regCond0}
Assume that $N\in \N_+$. $Z\coloneqq\{Z(s): s\in \Z^N \}$ is an $\R^d$-valued random field. The random variables $Z(s)$ are identically distributed and admit a bounded density $f$ w.r.t.\ the Lebesgue measure. Furthermore,

\begin{enumerate}[label=\textnormal{(\arabic*)}]

\item $Z$ is strongly spatial mixing with exponentially decreasing mixing coefficients, i.e., $\alpha(k) \le c_0 \exp( -c_1 \, k)$ for all $k\in\N_+$ for certain $c_0,c_1\in\R_+$. \label{Cond_Mixing}
\vspace{-.5em}
\item  Define the index sets by $I_{n} \coloneqq \{ s\in \Z^N: e_N\le s\le n\} \subseteq \N_+^N$ for $n\in\N^N$. All index sets considered in the following satisfy $\min\{ n_i: 1\le i \le N\} \ge C' \max\{ n_i: 1\le i\le N\}$ for a fixed constant $C'\in\R_+$. 
	\label{Cond_Sequence}
	
\item \label{DensityRequirement}
Let $a\in\N_+$ and denote the joint density of $Z(s)$ and $Z(t)$ by $f_{Z(s),Z(t)}$. There are two bounded and non-increasing radial functions $h, \tilde{h}: \R^d\to \R_{\ge 0}$ such that $f \le h$ and $|f_{Z(s),Z(t)}(z_1,z_2) - f(z_1)f(z_2) | \le \tilde{h}(z_1)\tilde{h}(z_2)$ for all $Z(s),Z(t)$, $s,t\in\Z^N$. Moreover,
$$
  \norm{h^{1/(2a)} }_{L^1(\lambda^d)}  < \infty \text{ and } \norm{ \tilde{h}^{1/a} }_{L^2(\lambda^d)}  < \infty.
$$

\end{enumerate}
\end{condition}

The assumption of exponentially decreasing $\alpha$-mixing coefficients is common, cf. \cite{LiWavelets}. One can show that such a rate is guaranteed for time series under mild conditions, cf. \cite{Withers1981} or \cite{davydov1973mixing}. 

The requirement on the constant $C'$ is technical. If we consider a sequence $(n(k):k\in\N) \subseteq \N^N$, where one coordinate $n_i(k)$ tends to infinity, then all other coordinates tend to infinity as well. This will also prove helpful in the following results, where we express the rates of convergence of the estimators in terms of the cardinality of the index set $I_n$, which is $|I_n|=\prod_{i=1}^N n_i$. For instance, if we obtain a rate of convergence in $\cO( |I_n|^{-\rho} )$ for a certain $\rho>0$, then this also means in terms of a single coordinate $i$ that the rate is in $\cO( n_i^{-N\rho})$. This reminds more of the case of i.i.d.\ or time series data, where we usually have observations $Z_1,\ldots,Z_n$. We do not require for our results an asymptotic on the index sets of the kind $I_{n(k)}\subseteq I_{n(k+1)}$ for a sequence $(n(k))$ in $\N^N$, we only require that all ratios $n_i / n_j$ are at least $C'$.

The condition on the function and the joint densities of the variables $Z(s)$ and $Z(t)$ is technical. The fact that the density $f$ is dominated by a radial function $h$, which satisfies certain integrability conditions, ensures that the tail of the density is well behaved. This is necessary for density functions with an unbounded support. If the density function is bounded and has bounded support, this condition is trivially satisfied.

The second requirement on the joint distribution of the random variables $Z(s)$ and $Z(t)$ restricts the mutual dependence, i.e., $f_{Z(s),Z(t)}(z_1,z_2) \le f(z_1)f(z_2) + \tilde{h}(z_1) \tilde{h}(z_2)$ for another radial function $\tilde{h}$ which satisfies certain integrability conditions. If $Z$ is strictly stationary this condition reduces to the joint densities of the pairs $( Z(0), Z(s) )$. Moreover, as the mixing coefficients vanish with increasing distance, we expect the dominating function $\tilde{h}$ to be determined by the pairs $(Z(s),Z(t))$ where $\norm{s-t}_\infty$ is small.

We can now define the density estimators. The density $f$ has the representation
\[
		f = \sum_{\gamma\in\Z^d} \theta_{0,\gamma}\,\Phi_{0,\gamma} + \sum_{k=1}^{|M|-1} \sum_{j = 0}^{\infty} \sum_{\gamma\in\Z^d} \upsilon_{k,j,\gamma} \,\Psi_{k,j,\gamma},  \text{ where } \theta_{j,\gamma} \coloneqq \scalar{f}{ \Phi_{j,\gamma} } \text{ and } \upsilon_{k,j,\gamma} \coloneqq \scalar{f}{ \Psi_{k,j,\gamma} }.
\]
Define the $j$-th approximation of $f$ by	$P_j f \coloneqq \sum_{ \gamma\in \Z^d } \theta_{j,\gamma} \, \Phi_{j,\gamma}$ for $j\ge 0$. Denote the linear estimator of $f$ given the sample $\{Z(s): s\in I_n\}$ by
\begin{align}
		\tilde{P}_j f &\coloneqq \sum_{ \gamma\in \Z^d } \hat{\theta}_{j,\gamma}\, \Phi_{j,\gamma}, \text{ where } \hat{\theta}_{j,\gamma} \coloneqq  |I_{n}|^{-1} \sum_{s\in I_{n}}\, \Phi_{j,\gamma}(Z(s)). \label{MRAApproxEmp}
\end{align}
The hard thresholding estimator is defined for two levels $0\le j_0 \le j_1$ and a thresholding sequence $(\bar{\lambda}_j:j\in\N)\subseteq \R_+$ as follows
\begin{align}\begin{split}
		\tilde{Q}_{j_0,j_1} f &\coloneqq \sum_{\gamma\in\Z^d}  \hat{\theta}_{j_0,\gamma}\,  \Phi_{j_0,\gamma} + \sum_{k=1}^{|M|-1} \sum_{j = j_0}^{j_1-1} \sum_{\gamma\in\Z^d}  \hat{\upsilon }_{k,j,\gamma}\; \1{  |\hat{\upsilon }_{k,j,\gamma} | > \bar{\lambda}_{j} }\;  \Psi_{k,j,\gamma},  \label{MRAApproxEmpHard}
\end{split}\end{align}
where $\hat{\upsilon }_{k,j,\gamma} \coloneqq  |I_{n}|^{-1} \sum_{s\in I_{n}}\, \Psi_{k,j,\gamma} (Z(s))$. Hence, $\tilde{Q}_{j_0,j_1} f$ consists of a linear estimator w.r.t.\ the coarse level $j_0$ and nonlinear terms of higher frequencies which are added to allow for more details if these are significantly different from zero. This also allows the approximation error and the estimation error of the estimator to vanish at higher rates than the linear estimator for certain parameter constellations, we encounter this below when presenting the results.

As $\tilde{P}_j f$ and $\tilde{Q}_{j_0,j_1} f$ are not necessarily a probability density, one can additionally consider the normalized estimator. We refer to Appendix~\ref{Appendix_QuestionOfNormalization} for this question.

\section{Linear and hard thresholded wavelet density estimation}\label{WaveletResults}
In this section we study wavelet density estimators for $d$-dimensional data. \cite{KellyWaveletExpansions} show that for isotropic wavelets and $f\in L^{p'}(\lambda^d)$ ($1\le p' <\infty$) the approximation bias vanishes, $\norm{f-P_j f}_{L^{p'}(\lambda^d)} \rightarrow 0$ as $j\rightarrow \infty$. In the case $p'=\infty$ it is not guaranteed that the approximation error vanishes for general elements from $L^{p'}$: for instance, consider the one-dimensional Haar mother wavelet $\psi \coloneqq \I\{[0,1/2) \} - \I\{[1/2,1)\}$ and construct with it the density $f \coloneqq \I\{[0,1)\} + \sum_{j=0}^{\infty} \psi\left(2^{j+1}x - (2^{j+1}-2) \right)$ on the unit interval $[0,1]$. $f$ jumps between 0 and 1 and these jumps become quite erratic as $x$ tends to 1. In particular, the projection $P_j f$ onto $U_j$ cannot capture all jumps. Hence, we have $\liminf_{j\rightarrow \infty} \norm{f-P_j f}_{\infty} \ge \frac{1}{2} > 0$ and the approximation property fails in this case. However, if $f$ is a Besov density in $B^{s}_{p,q}(\R^d)$, we can derive for general admissible matrices $M$ a rate of convergence.

We begin with the linear estimator, the technique of the proof is based on the ideas of \cite{kerkyacharian1992density} who consider the case for one-dimensional i.i.d.\ data. 

\begin{theorem}[Rate of convergence of the linear estimator]\label{LinearDensityEstimationBesovFunction}
Let $p' \in [1,\infty)$, $p,q\in [1,\infty]$ and $s >1/p$. Define $s'\coloneqq s + (1/p' - 1/p) \wedge 0$. Let $A\in\cB(\R^d)$ and if $p'<p$, assume moreover that $A$ is bounded. 

Let $f \in F_{s,p,q}(K,A)$ for some $K\in\R_+$. If $p'\in[1,2]$, assume moreover that Condition~\ref{regCond0}~\ref{DensityRequirement} is satisfied with $a=1$. If $p'\in(2,\infty)$, assume furthermore that Condition~\ref{regCond0}~\ref{DensityRequirement} is satisfied with $a=2$. Let the level $j$ grow at a rate of $	|M|^j \simeq	|I_n|^{1/ (2s' + 1) } $. Then 
$
		\E{ \int_{\R^d} |f - \tilde{P}_j f |^{p'}}^{1/p'} \le	C_1 \, |I_n|^{- s'/(2s'+ 1) }
$
for suitable constant $C_1\in\R_+$.

Moreover, if the domain $A$ is bounded and $K>0$ is a fixed constant, then
\begin{align}\label{LinearWaveletEstimationRates}
	 \sup_{f \in F_{s,p,q}(K,A) } \E{ \int_{\R^d} |f - \tilde{P}_j f |^{p'}}^{1/p'} &\le 
	C_2 \, |I_n|^{- s'/(2s'+ 1) }
\end{align}
for all dependence structures of the random field $Z$, which satisfy $|f_{Z(s),Z(t)}(z_1,z_2) -f(z_1)f(z_2)| \le h_0$ for all $s,t\in\Z^N$ for some fixed $h_0\in\R_+$.

The constants $C_1, C_2$ depend on the wavelets $\Psi_k$ ($k=0,\ldots,|M|$), the matrix $M$, the bound on the mixing rates, the domain $A$, the index $p'$. $C_1$ depends additionally on the functions $h$ and $\tilde{h}$. $C_2$ depends additionally on $K$.
\end{theorem}

\cite{kerkyacharian1992density} obtain with similar requirements for a real-valued sample i.i.d.\ sample $Z_1,\ldots,Z_n$ a rate of convergence which is in $\cO( n^{-s'/(2s'+1) } )$. 

\cite{bouzebda2015multivariate} study the linear wavelet estimator for multivariate time series $Z_1,\ldots,Z_n$ in the supremum norm and in the case where $M=2I$. They obtain a bound of $\cO(	2^{j d/2 } (\ln n)^{1/2}	/ n^{1/2} + 2^{-jd \rho } )$ for a level $j$ which depends on $n$. The first term is the estimation error, the second term is the approximation error (bias); this error depends on a certain smoothness parameter $\rho$. Thus, their rates are quite similar to our result in particular if we compare it to the intermediate result from Theorem~\ref{LpConvergence} in Section~\ref{Proofs_of_the_main_theorems} which considers the estimation error of the linear estimator.

Hence, when compared to the one-dimensional i.i.d.\ situation, we see that the estimate with a strongly mixing $d$-dimensional sample achieves the analogue rate which is $|I_n|^{-s'/(2s'+1)}$.

The data dimension $d$ is relevant for the rate of convergence, however, this is not shown in the previous theorem. We highlight this fact in the next two results which show that the dimension $d$ has a negative impact on the Besov parameter $s$ which controls the decay of the coefficients $\upsilon_{k,j,\gamma}$ as $j$ tends to infinity. We demonstrate that the classical inclusions shift slightly when moving from the one-dimensional to the $d$-dimensional Besov space: consider an $(A,r)$-H{\"o}lder continuous function w.r.t.\ the 2-norm, i.e.,	$|f(x)-f(y)| \le A \norm{x-y}_{2}^r$ for all $x,y\in\R^d$ for some $0<A<\infty$. In the one-dimensional case we have that $f$ belongs to the space $B^r_{\infty,\infty}$, i.e., that $\norm{f}_{r,\infty,\infty}$ is finite, see, e.g., \cite{donoho1996}.

However, in the multivariate case, we find that such a function can be embedded only in a Besov space $B^s_{\infty,\infty}$ for an $s<r$ which yields a slower rate of convergence. Consider a wavelet coefficient of $f$:
\begin{align*}
		|\upsilon_{k,j,\gamma}| &\le \left| \int_{\R^d} (f(x)-f(x_0)) \Psi_{k,j,\gamma}(x)\,\intd{x} \right| + |f(x_0)| \,  \left| \int_{\R^d} \Psi_{k,j,\gamma}(x)\,\intd{x} \right|  \\
		&\qquad\qquad\qquad\le \sup\left\{ |f(x)-f(x_0)| : x\in\supp{ \Psi_{k,j,\gamma} } \right\} \; |M|^{-j/2} \norm{\Psi_k }_1 \nonumber \\
		&\qquad\qquad\qquad\qquad\qquad\qquad\le A \left(L\sqrt{d} \norm{M^{-j}}_2 \right)^r \; |M|^{-j/2} \norm{\Psi_k}_1 \le C\, (\zeta_{min})^{-jr} \, |M|^{-j/2},
\end{align*}
where $\supp{\Psi_k} \subseteq [0,L]^d$ and the point $x_0 \in \supp{ \Psi_{k,j,\gamma} }$ is in the support of $\Psi_{k,j,\gamma}$ and $C\in\R_+$ is a suitable constant. Hence, we have for the $\norm{\,\cdot\,}_{s,\infty,\infty}$-norm of $f$ if $p=q=\infty$:
\[
		\sup_{k,j,\gamma}  |M|^{j(s+1/2)} \, |\upsilon_{k,j,\gamma}|  \le C \sup_{j} \, (\zeta_{max})^{jsd} \, (\zeta_{min})^{-jr} < \infty \text{ if }  s \le \frac{r}{d} \frac{ \ln \zeta_{min} }{ \ln \zeta_{max} }.
\]
Hence, if $f$ is an $(A,r)$-H{\"o}lder density and $s = r \ln \zeta_{min} / (d \ln \zeta_{max} ) \le r $, then $\norm{f}_{s,\infty,\infty} < \infty$. We see that the difference in the eigenvalues $\zeta_{min}$ and $\zeta_{max}$ has little impact as it only enters with the logarithm. However, far more relevant is the data dimension $d$ which scales the regularity parameter with its inverse $d^{-1}$.

One finds in simple examples that the bound of the first inequality is sharp: indeed, consider the Lipschitz function which is non-constant in only one coordinate, $f(x) \coloneqq x_1$ and use an MRA given by isotropic Haar wavelets. In this case, one computes
\[
		\sup_{k,j,\gamma} |M|^{j(s+1/2)} |\upsilon_{k,j,\gamma}| = \sup_{j} 2^{j(ds-1)} / 4 < \infty \text{ if and only if } s \le 1/d.
\]

Using this insight, we can formulate two statements which also reveal that with increasing data dimension $d$ the rate of convergence deteriorates.

\begin{corollary}[H{\"o}lderian densities]\label{HolderDensity}
Let $f$ be a compactly supported $d$-dimensional $(A,r)$-H{\"o}lderian density. The linear estimator attains the rate given in Equation~\eqref{LinearWaveletEstimationRates} for the parameter choice $s' = s = r \ln \zeta_{min} / (d \ln \zeta_{max} )$.
\end{corollary}

The next result also applies to density functions with unbounded support

\begin{theorem}[Differentiable densities]\label{DifferentiableDensity}
Let ${p'}\in [ 1, \infty)$. If $p'\le 2$ (resp. $p'>2$), assume moreover that Condition~\ref{regCond0}~\ref{DensityRequirement} is satisfied with $a=1$ (resp. $a=2$). Additionally, the gradient of $f$ is bounded by a non-increasing radial function $\bar{h}\in L^{p'}(\lambda^d)$, i.e., $\norm{ Df }_2 \le \bar{h} $. Choose 
$
	j  \coloneqq j_0 + \floor*{ ( 2\ln \zeta_{min} + d\ln \zeta_{max} )^{-1} \, \ln |I_n| },
$
for a $j_0\in\N$. Then the linear estimator attains the rate	$ \cO\left( |I_n|^{-\ln \zeta_{\min}/ ( 2 \ln \zeta_{\min} + d \ln \zeta_{\max})}			\right)$.
\end{theorem}

Next, we study the nonlinear hard thresholding estimator of \cite{donoho1996} who consider this estimator for one-dimensional i.i.d.\ data. It is well-known that hard thresholding preserves the visual appearance of jumps and peaks of the density.
 A short and heuristic motivation for this estimator is as follows: the proof of Theorem~\ref{LinearDensityEstimationBesovFunction} (in particular, Theorem~\ref{LpConvergence} and Equation~\eqref{ApproximationErrorEq}) reveals that the bias (approximation error) of the linear estimator is of order $\cO(|M|^{-js'})$ while the stochastic term (estimation error) is of order $|M|^{j/2} |I_n|^{1/2}$. This cannot be optimal for a density $f\in B^s_{p,q}$ if $p'>p$ because in this case the bias is of the wrong order as $s' < s$, see \cite{donoho1996} for a deeper discussion and more details.  However, if the density $f$ belongs to a Besov space $B^s_{p,q}$, then this restriction entails that many coefficients $\upsilon_{k,j,\gamma}$ are forced to decay at a high rate: in particular, the decay in the $\ell^p$-sequence norms $\norm{\upsilon_{k,j,\cdot}}_{\ell^p}$ has to overcompensate the exponential growth of $|M|^{j(s+1/2-1/p)}$. Consequently, it makes sense to add finer levels of the density (more details) such that the bias is again of the right order and to set insignificant estimates of these details, $\hat{\upsilon}_{k,j,\gamma}$, to zero.

\begin{theorem}[Rate of convergence of the hard thresholding estimator]\label{LPconvergenceHardThresholding}
Let $p' \in [1,\infty)$, $p,q\in [1,\infty]$ and $s >1/p$. Let $A\in\cB(\R^d)$ and if $p'<p$, assume moreover that $A$ is bounded. Let Condition~\ref{regCond0}~\ref{DensityRequirement} be satisfied with $a=4$.

Let $f \in F_{s,p,q}(K,A)$ for some $K\in\R_+$. Set $\bar{\lambda}_j \coloneqq K_0 \sqrt{j/|I_n|}$ for a constant $K_0$ specified in \eqref{hardThresChoiceK0}. Define
$\epsilon = sp - (p' - p)/2$ and $s' = s + (1/p' - 1/p) \wedge 0 $.
 Moreover, define the levels $j_0$ and $j_1$ as
\begin{align*}
 &|M|^{j_0} \simeq |I_n|^{1-2\alpha} \text{ and } |M|^{j_1} \simeq |I_n|^{\alpha/s'}, \text{ where }
	\alpha \coloneqq \begin{cases} 
		\frac{s}{2s+1} &\text{ if } \epsilon \ge 0 \\
		\frac{s'}{2s+1-2/p} &\text{ if } \epsilon < 0.
		\end{cases}
\end{align*}
Then
$
	  \E{ \int_{\R^d} |f - \tilde{Q}_{j_0,j_1} f |^{p'} \intd{\lambda}^d }^{1/p'} \le C_1 |I_n|^{-\alpha} \left( \ln |I_n|\right)^{(3p'-p)/2p' \cdot \1{p' > p }}.
$

Moreover, if the domain $A$ is bounded and $K>0$ is a fixed constant, then
\begin{align*}
	 \sup_{f \in F_{s,p,q}(K,A) } \E{ \int_{\R^d} |f - \tilde{Q}_{j_0,j_1} f |^{p'}}^{1/p'} &\le 
	C_2 \, |I_n|^{-\alpha} \left( \ln |I_n|\right)^{(3p'-p)/2p' \cdot \1{p' > p } }.
\end{align*}
for all dependence structures of the random field $Z$, which satisfy $|f_{Z(s),Z(t)}(z_1,z_2) -f(z_1)f(z_2)| \le h_0$ for all $s,t\in\Z^N$ for some fixed $h_0\in\R_+$.

The constants $C_1, C_2$ depend on the wavelets $\Psi_k$ ($k=0,\ldots,|M|$), the matrix $M$, the bound on the mixing rates, the domain $A$, the index $p'$. $C_1$ depends additionally on the functions $h$ and $\tilde{h}$. $C_2$ depends additionally on $K$.
\end{theorem}

Note that the exponent $\alpha$ is smooth in the parameters $s,p,p'$, i.e., for the case $\epsilon = 0$, we could also use the definition of the case where $\epsilon<0$. This also means that the rate of convergence is smooth in the parameter $p'$.

We see that the rates have the identical structure than the rates of \cite{donoho1996} who consider the classical case for a one-dimensional density and i.i.d.\ data. If $p' \le p $, then $\epsilon>0$ and we obtain that $j_1$ and $j_0$ grow at the same rate. So the linear estimator is the preferred choice. If $p' > p$, however, one computes that in each case the exponent $\alpha$ is strictly greater than $s'/(2s'+1)$, the latter is the exponent which determines the rate of the linear estimator. Consequently, the nonlinear estimator performs better in this case. 

\cite{LiWavelets} studies the hard thresholding estimator in the special case $p'=2$ for random fields similar as we do, however, the data are one-dimensional. He obtains for a one-dimensional density $f\in F_{s,p,q}(K,[-A,A])$ a rate in $L^2$ of $\cO( ( \ln |I_n| / |I_n| )^{s/(2s+1)} )$. So our results are a generalization as we do not only consider general $p'$ but also allow for multivariate data and nonisotropic wavelets. Moreover, if $p' \ge  p$, the density function can have an unbounded support .

An alternative to hard thresholding is soft thresholding of the coefficients. Here the absolute value of the estimated coefficients $\nu_{k,j,\gamma }$ undergoes the nonlinear shrinkage process $x\mapsto \operatorname{sgn} x \cdot (x-\delta)_+$ for a certain $\delta>0$.

This procedure can be interpreted as suppressing the noise in the estimated coefficients. Hence, one could also investigate the soft thresholding density estimator in the present setting. Fundamental properties of the soft thresholding method have been investigated by \cite{donoho1997universal}. \cite{delouille2001nonparametric} study the soft and hard thresholding estimator for design-adapted wavelets in nonparametric regression for one-dimensional i.i.d.\ data. They obtain a rate of convergence in $L^2$ which is in $\cO( (\ln n /n)^{r/(2r+1)} )$ if the regression function is H{\"o}lder continuous with exponent $1/2 < r < 1$. This corresponds to our findings for H{\"o}lder continuous densities.

\section{Numerical results}\label{ExamplesOfApplication}

We give an example for the estimation of a two-dimensional density with strongly spatial mixing sample data on a regular two-dimensional lattice. Hence, concerning the parameter choice from the previous sections, we have a lattice dimension $N$ equal to 2 and two-dimensional data, i.e., $d=2$. The section is divided in three parts. Firstly, we describe our decision rule how to choose the tuning parameters $\lambda$, $j_0$ and $j_1$. Secondly, we sketch the process which generates the sample data on the lattice. And finally, we present some numerical results for the estimation of a selected density function.

We follow a simple validation approach in order to choose the tuning parameters. We do not use leave-one out cross-validation because we face a dependent sample and cross-validation could corrupt the inner stochastic structure. In what follows $I_n$ is a finite and rectangular subset of $\N_+^2$. We can construct a graph $G$ from this set if we use the four-nearest-neighborhood structure on $\Z^2$ to construct the edge set. This neighborhood structure will define the dependence between the random variables $\{ Z(s): s\in I_n \}$ which all have the same marginal density $f$. The index set $I_n$ is partitioned into two sets $I_{n,1}$ and $I_{n,2}$. Here we assume that each index set is a connected set w.r.t.\ the four-nearest neighborhood structure.

The density is estimated from the sample data which belongs to the index set $I_{n,1}$, we denote the estimate by $\hat{f}_n $. As in the present example $I_n = I_{(n_1,n_2)} = \{ s: 1\le s_i \le n_i, i=1,2 \}$, we choose $I_{n,1} $ as $\{ s: 1\le s_i \le \floor{0.9 n_i}, i=1,2 \}$. So $I_{n,1}$ is also a rectangular set and the estimator is computed with approximately 80 \% of the data.

In general, the integrated squared error can be decomposed as
\begin{align}\label{DefISE}
			\operatorname{ISE}(f, \hat{f}_n ) &= \int_{\R^d} (\hat{f}_n  - f )^2 \intd{\lambda^d} = \left\{ \int_{\R^d} \hat{f}_n^2 \intd{\lambda^d } - 2 \int_{\R^d} \hat{f}_n \, f \; \intd{\lambda^d } \right\} + \int_{\R^d} f^2 \intd{\lambda^d} \nonumber \\
			&= \operatorname{Ver}(\hat{f}_n, f) + \norm{f}^2_{L^2(\lambda^d)}.
\end{align}
Since in practice the true density function is unknown, it is sufficient for a comparison of density estimates to compute the full validation criterion with the subsample $I_{n,2}$, which is L-shaped in the present example. We define
\begin{align}
		\widehat{\operatorname{Ver}}(\hat{f}_n, f, I_{n,2} ) &\coloneqq \int_{\R^d} \hat{f}_n^2 \intd{\lambda^d } - 2 \frac{1}{|I_{n,2}|} \sum_{s\in I_{n,2}} \hat{f}_n (Z(s)),\label{validationCriterion}
\end{align}
which is the empirical analogue of $\operatorname{Ver}(\hat{f}_n, f)$ as on average we expect only a small dependence between $\hat{f}_n$ and the random variables $Z(s)$ for $s\in I_{n,2}$.

For hard thresholding, we use an approach similar to an algorithm which has been proposed by \cite{hall2001} for the choice of the primary level $j_0$ in the context of cross-validation. The idea is to define a suitable partition $R_1 \cup ... \cup R_S$ of the domain of definition of $\hat{f}_n$ (resp. of $f$), where each $R_k$ collects regions of relatively homogenous roughness. These regions can be determined with a pilot estimator. We compute the validation criterion for each $R_k$ for the levels $j = j_0,\ldots, j_1$ ($j_0 \le j_1$) with the purely linear wavelet estimator $\tilde{P}_j f$ from Equation \eqref{MRAApproxEmp} restricted to $R_k$. Abbreviate the level which minimizes \eqref{validationCriterion} for region $R_k$ by $j_k$. Then choose $j^* \coloneqq \min\{j_k: k=1,\ldots,S\}$ as the primary level.

Moreover, we follow an approach used in \cite{hardle2012wavelets} for the hard thresholding estimator from \eqref{MRAApproxEmpHard} and set each threshold as a multiple of $\max\{ |\hat{\upsilon}_{k,j,\gamma}| :  k=1,\ldots,|M|-1,\gamma\in\Z^d \}$ for $j = j^*,\ldots,j_1$. This multiple is the same for all $j=j^*,\ldots,j_1$. In the following, we refer to these multiples as the threshold, e.g., a threshold 0.1 means that $\bar{\lambda}_j$ is equal to 0.1 times $\max\{ |\hat{\upsilon}_{k,j,\gamma}| :  k=1,\ldots,|M|-1,\gamma\in\Z^d \}$.

Next, we sketch the data generating process. We use the algorithm of \cite{kaiser2012} to simulate five random vectors $Z_1,\ldots,Z_5$. The marginals of each vector are standard normally distributed. The underlying graph is the same for each random vector and is the regular two-dimensional lattice with the four-nearest-neighborhood structure and edge lengths $n_1 = n_2$. We perform the simulation for four different values of $n_1$, namely 20, 35, 50 and 65. So that $|I_n|$ is 400, 1225, 2500 and 4225. 

The dependence within a random vector $Z_i$ is generated as follows. Write $Z$ for this vector for simplicity. If $Z = \{ Z(s): \, s\in I_n \}$ is multivariate normal with expectation $\alpha \in \R^{|I_n|}$ and covariance $\Sigma \in \R^{ |I_n|\times |I_n|}$, then $Z$ has the density
\[
		f_Z( z) = (2\pi)^{-d/2 } \text{det}(\Sigma)^{- 1/2 } \exp \left(	-\frac{1}{2} (z-\alpha)^T \Sigma^{-1} (z-\alpha) \right).
\]
Using the notation $P$ for the precision matrix $\Sigma^{-1}$ and $-s \coloneqq I_n\setminus\{s\}$, we have
\[
		Z(s) \, |\, Z(-s) \sim \cN\left(  \alpha(s) - (P(s,s))^{-1}  \sum_{t \neq s} P(s,t) \Big( z(t) - \alpha(t) \Big)  , (P(s,s))^{-1} \right).
\]
Write $\operatorname{Ne}(s)$ for the neighbors of $s$ in $I_n$ w.r.t.\ the four-nearest-neighborhood structure. Since $P = \Sigma^{-1}$ is symmetric and since we can assume that $\left(P(s,s) \right)^{-1} > 0$, $Z$ is a Markov random field if and only if for all nodes $s\in I_n$
\begin{align*}
		P(s,t) \neq 0 \text{ for all } t \in \operatorname{Ne}(s) \text{ and }
		P(s,t) = 0 \text{ for all } t \in I_n \setminus \operatorname{Ne}(s).
\end{align*}
\cite{cressie1993statistics} investigates the conditional specification
\begin{align}\label{ConditionalNormalCressie}
		Z(s) \,|\, Z(-s) \sim \cN\left( \alpha(s) + \sum_{t \in \operatorname{Ne}(s)} c(s,t) \big(Z(t) - \alpha(t) \big),  \varsigma^2(s) \right)
\end{align}
where $C=\big( c(s,t) \big)_{s,t}$ is a $|I_n|\times |I_n|$ matrix and $D = \text{diag}(\varsigma^2(s): s\in I_n)$ is a diagonal matrix such that the coefficients satisfy the necessary condition $\varsigma^2(s) \, c(t,s) = \varsigma^2(t)\, c(s,t)$ for $s\neq t$ and $c(s,s) = 0$ as well as $c(s,t) = 0 = c(t,s)$ if $s,t$ are no neighbors. This means $P(s,t) = -c(s,t)\, P(s,s)$, i.e., $\Sigma^{-1} = P =  D^{-1}  (I-C)$. If $I-C$ is invertible and $(I-C)^{-1} D$ is symmetric and positive definite, then the entire random field is multivariate normal with	$Z \sim \cN\left( \alpha, (I-C)^{-1} D \right)$.

In particular, it is plausible in many applications to use equal weights $c(s,t)$: we can write the matrix $C$ as $\eta H$, where $H$ is the adjacency matrix of $G$, i.e., $H(s,t)$ is 1 if $s,t$ are neighbors, otherwise it is 0. Denote the minimal eigenvalue of $H$ by $h_0$ and the corresponding maximal eigenvalue by $h_m$. We know from the properties of the Neumann series that $I-C$ is invertible if $(h_0)^{-1} < \eta < (h_m)^{-1}$ in the case where $h_0<0<h_m$; this last condition is often satisfied in applications. We choose the conditional variance $\varsigma^2(s)$ such that the diagonal matrix $D$ consists of the inverse elements of the diagonal of the matrix $(I-C)^{-1}$. Hence, the marginals of the $Z(s)$ are standard normally distributed. 

The graph structure of the index set $I_n$ allows us two partition $I_n$ into two sets $C_1$ and $C_2$, which are disjoint w.r.t.\ the edges from the four-nearest neighborhood structure such that within each set $C_i$ any two points $s,t$ are no neighbors. The sets $C_i$ are termed concliques, see \cite{kaiser2012}. An important property is now the following: if $Z(s)$, $s\in I_n$ is a Markov random field, then the conditional distribution of the $Z(s)$ with $s\in C_1$ given the $Z(s)$ with $s\in C_2$ factorizes as a product due to the conditional independence. The same is true if we change the roles of $C_1$ and $C_2$.

This insight yields the following MCMC algorithm: All $Z(s)$, $s\in I_n$ are initialized according to a certain distribution. Then with the help of Equation~\eqref{ConditionalNormalCressie} conditional on the conclique $C_2$ the random variables $Z(s)$ with $s\in C_1$ are updated. Afterwards, the random variables which belong to $C_2$ are updated with \eqref{ConditionalNormalCressie} based on the (new) realizations of $C_1$. The last two steps are repeated many times until the random field approximately reaches its stationary distribution. Hence, if we compare this method to the Gibbs sampler, we see that a complete update of the random field can be performed in two steps. More details on this procedure, in particular its asymptotic properties, can be found for instance in \cite{krebs2017orthogonal}.

As we do not simulate one random vector but instead five, we use a Gaussian copula in the update steps such that also four of the random vectors are dependent, namely $Z_1$, $Z_2$, $Z_3$ and $Z_4$. $Z_5$ is independent of the first four. So we have a dependence within each component $Z_i$ and the first four components are also dependent among each other.

We run 1000 iterations of the MCMC-algorithm for the simulation of $(Z_1,\ldots,Z_5)$. The parametrization of the multivariate normal distribution is chosen as follows $\alpha_i(s) \equiv 0$ and $\sigma_i = 1$ for all $s \in I_n$ and $i=1,\ldots,5$. The dependence parameters $\eta_i$ that determine the interaction within a distribution $Z_i$ are chosen as follows 0.2, -0.1, -0.22, 0.2 and 0.22, note that $|\eta_i| = 0.22$ constitutes a strong dependence, whereas $\eta_i = 0$ indicates independence. In this case the admissible range of $\eta$ is very close to $(-0.25, 0.25)$ which is the parameter space of $\eta$ for a lattice wrapped on a torus. The approximate correlations of the first four $Z_i$ are given by	$	\rho_{1,2} \approx 0.1, \rho_{1,3} \approx 0, \rho_{1,4} \approx 0, \rho_{2,3} \approx 0, \rho_{2,4} \approx 0 \text{ and } \rho_{3,4} \approx 0.1$.

With these distributions we define a random variable $Y$ with a non-continuous density as follows: firstly, we retransform $Z_5$ to a discrete random variable $S$ which takes the states $0$ and $1$ with probability $1/2$. Secondly, transform $Z_1$ and $Z_2$ to a random variable $U_1$ and $U_2$ which are both uniformly distributed on $[0,1]$. And thirdly, we define $X_1$ and $X_2$ as rescaled and shifted $Z_3$ and $Z_4$ such that they are normally distributed with parameters $\mu = 0.5$ and $\sigma^2 = 0.2$. Set now $Y = \I\{ S = 0\} \, [U_1,U_2] + \I\{ S = 1\} \, [X_1, X_2]$, then $Y$ admits the density
\[
		f_{(Y_1,Y_2)} = \frac{1}{2} \, 1_{[0,1]^2 } + \frac{1}{2}\, \cN\left( \begin{pmatrix} 0.5\\ 0.5 \end{pmatrix},  0.2^2 \begin{pmatrix} 1 &  \rho \\ \rho & 1 \end{pmatrix} \right),
\]
where $\rho\approx 0.1$. A density plot is given in Figure~\ref{fig:Density_True_Haar_D4_0}. We estimate the density $f_{(Y_1,Y_2)}$ with the linear and the nonlinear wavelet estimators based on isotropic Haar wavelets and Daubechies 4-wavelets and the sample $Z(s)$ $s\in I_{n,1}$. We abbreviate the Daubechies wavelet by $D4$ (resp. $db2$), see \cite{daubechies1992ten}.

Then we compute for several levels the verification criterion from Equation~\eqref{validationCriterion} with the random variables $Z(s)$ $s\in I_{n,2}$. We perform this whole procedure $1000$ times in total for each sample size. This means that we compute for each of the 1000 simulations the verification criterion for each estimator at each level $j$. Afterwards we can compute the average and the empirical standard deviation of this statistic for each scenario. The numerical results are given in Table~\ref{TableExample2}. Here the verification criterion is computed for different levels $j$ reaching from 0 to 4. This means for the linear estimator that the coefficients $\theta_{j,\gamma}$ are computed for each of these levels. For the nonlinear estimator the coefficients $\theta_{j,\gamma}$ are computed only for the level $j=0$. Then the nonlinear details $\upsilon_{k,j,\gamma}$ are added successively for $j=1,\ldots,4$. We also show in the table which estimator, which levels $j$ and which threshold are the most suitable for each sample size.

For a better comparison, we also give in Table~\ref{TableExample2_ind_ref_sample} the results, which are derived with an independent reference sample $\widetilde{Z}=(\widetilde{Z}_1,\ldots,\widetilde{Z}_5)$. This means that the random variables within a component $\widetilde{Z}_i$ are i.i.d., i.e., $\widetilde{Z}_i(s)$ are i.i.d.\ for $s\in I_n$ and for fix $i=1,\ldots,5$. The correlations between the vectors $\widetilde{Z}_i$ correspond to those of the $Z_i$. Note that we use for hard thresholding several multiples of $\max\{ |\hat{\upsilon}_{k,\ell,\gamma}| :  k=1,\ldots,|M|-1, \gamma\in\Z^2 \}$, however, the multiple is the same for all levels $j^*,\ldots,j_1$ and only varies for the entire estimator. Examples of density estimates are given in Figures~\ref{fig:Density_Haar} and \ref{fig:Density_D4}. The estimators have been corrected for possible negative regions, we refer to Appendix~\ref{Appendix_QuestionOfNormalization}.

Firstly, we see that the estimator from the dependent sample performs as well as the estimator from the independent sample, as suggested by the theoretical results. This is true for each sample size and for each wavelet type. Moreover, we find that the optimal multiple of the threshold $\bar{\lambda}$ is the same for all sample sizes. In particular it is different from zero in each case. 

We note that the values of the verification criterion can be compared across the different wavelets, this follows from Equations~\eqref{DefISE} and \eqref{validationCriterion}. Hence in terms of the validation criterion, we find that the Daubecchies wavelet performs better than the Haar wavelet for each sample size in this example. However, they have in common that the optimal level $j$ is the same for each sample size: the level $j=2$ minimizes the criterion for the sample sizes 400 and 1225. It is level $j=3$ for the larger sample sizes 2500 and 4225.

\begin{figure}[ht]
    \centering
    \begin{subfigure}[b]{0.33\textwidth}
				\includegraphics[height = 6.7cm,  keepaspectratio, trim = 0 0 0 0, clip=true]{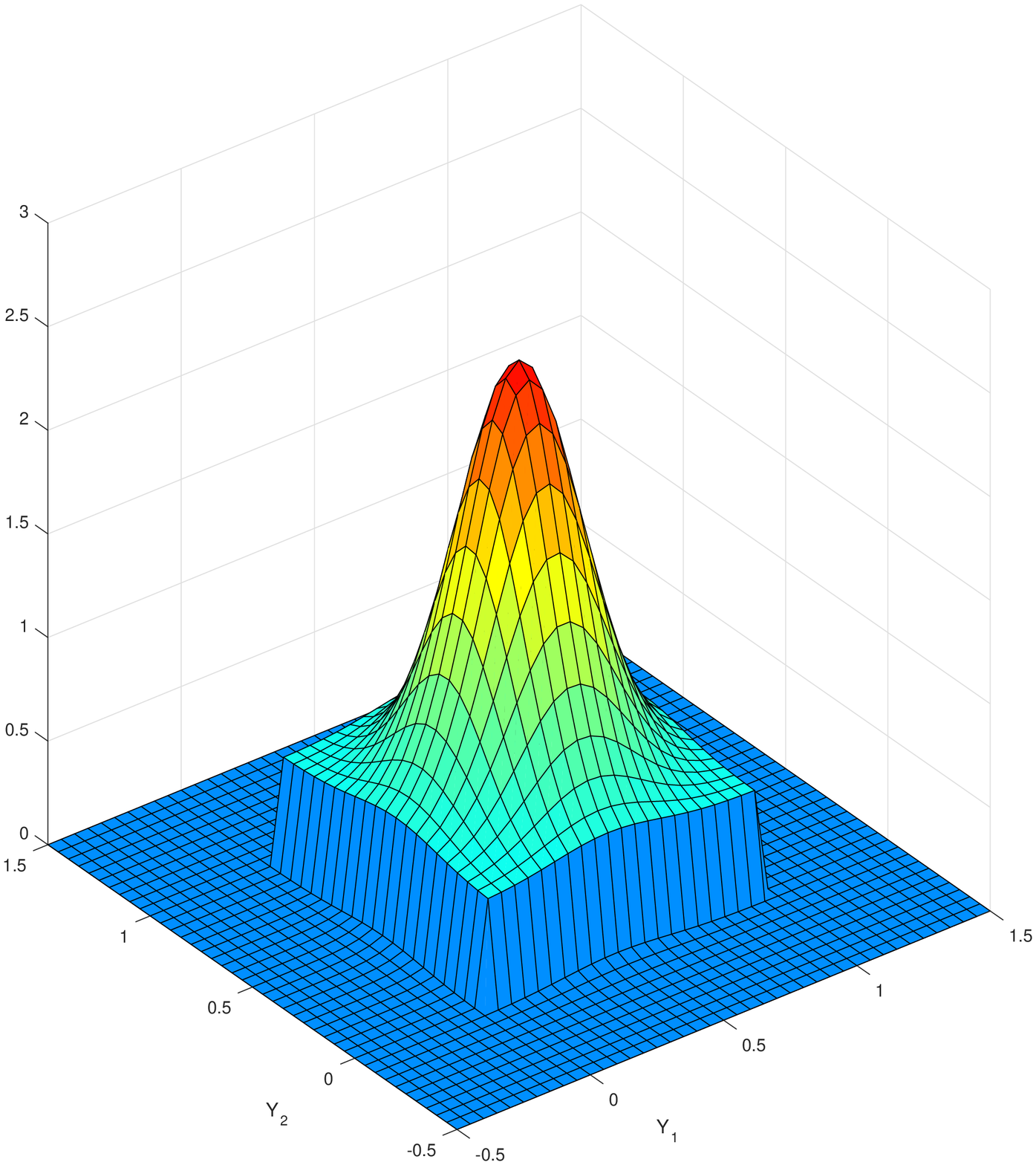}
				\caption{True density function}
					\label{fig:Density_True_Haar_D4_0}
    \end{subfigure}
    \begin{subfigure}[b]{0.33\textwidth}
        \includegraphics[height = 6.7cm, keepaspectratio, trim = 0 0 0 0, clip=true]{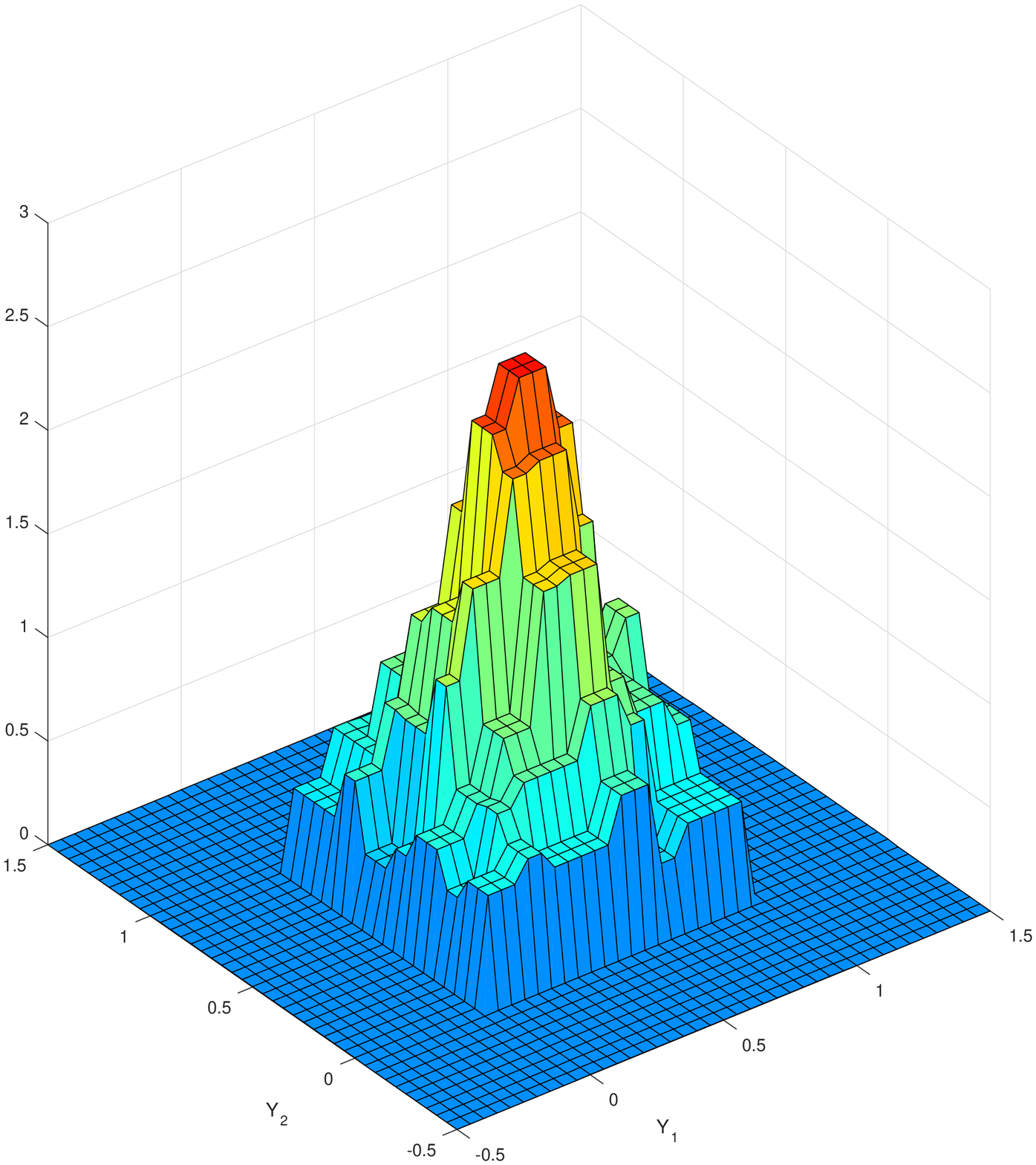}
				\caption{Haar estimate (for $j=3$, $\lambda = 0.1$)}
				\label{fig:Density_Haar}
    \end{subfigure}
    \begin{subfigure}[b]{0.33\textwidth}
        \includegraphics[height = 6.7cm,  keepaspectratio, trim = 0 0 0 0, clip=true]{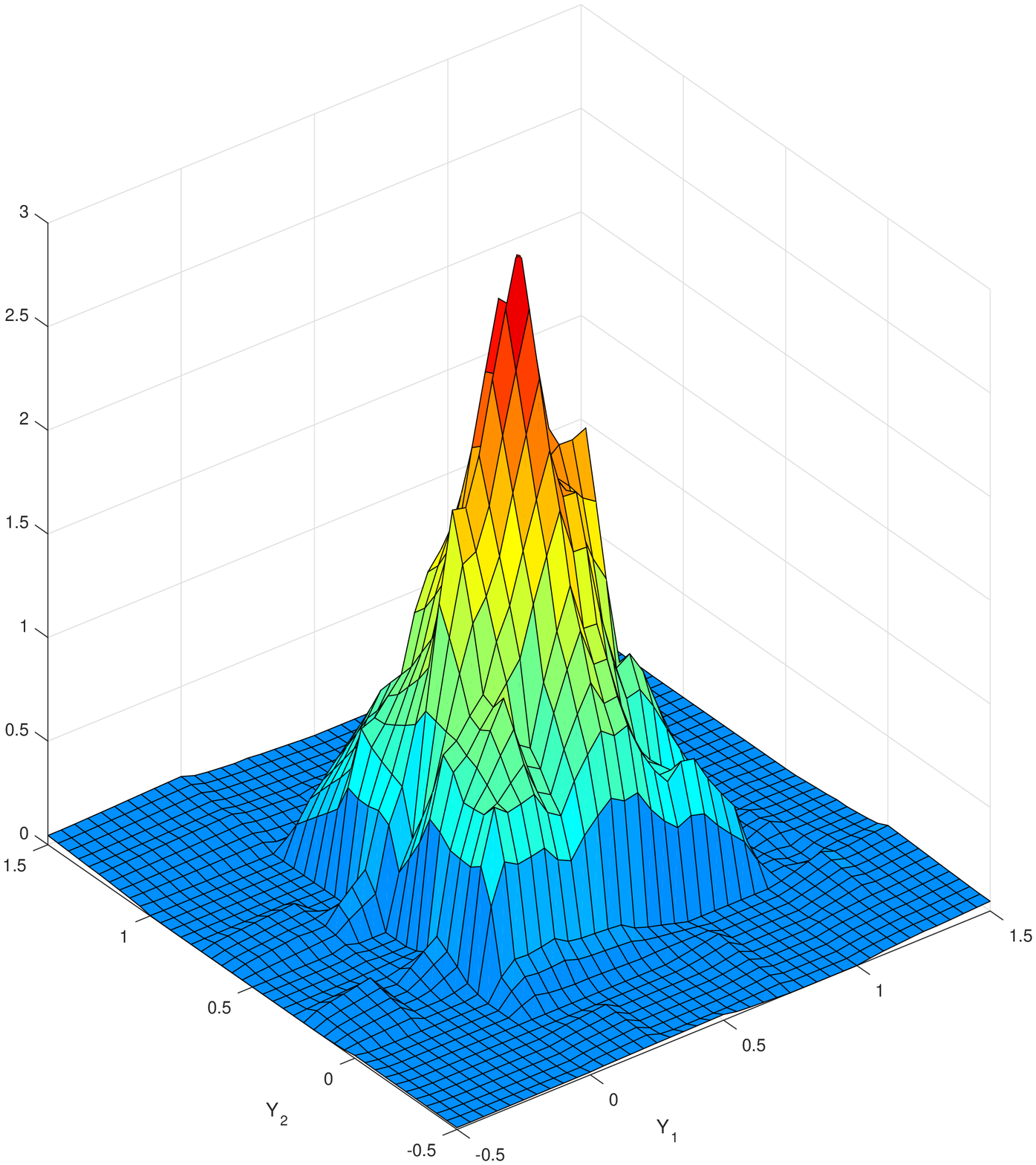}
				\caption{D4 estimate (for $j=3$, $\lambda = 0.1$)}
				\label{fig:Density_D4}
    \end{subfigure}
    \caption{Estimation of a mixture density with a sample from the two-dimensional lattice of size 4225.}\label{fig:DensityEstimationPics}
\end{figure}

\begin{table}[ht]
\begin{center}
\begin{tabular}{|| c || c || c || c | c | c || c || c | c | c ||  }
\hline
\hline
	& & \multicolumn{4}{|| c ||}{ Haar } & \multicolumn{4}{| c ||}{ D4} \\
\hline
\hline
sample size & j & linear & \multicolumn{3}{c ||}{ nonlinear: hard threshold } & linear & \multicolumn{3}{c ||}{ nonlinear: hard threshold } \\
\hline
	& &					&	0.1  & 0.2 & 0.3 &			&	0.1 & 0.2 & 0.3 \\
\hline
\hline
\multirow{2}{*}{400}	&	\multirow{2}{*}{0}	&	-0.922	&	-	&	-	&	-	&	-0.583	&	-	&	-	&	-	\\
	&		&	(0.040)	&	-	&	-	&	-	&	(0.054)	&	-	&	-	&	-	\\
\hline																			
\multirow{2}{*}{400}	&	\multirow{2}{*}{1}	&	-0.871	&	-0.922	&	-0.922	&	-0.922	&	-1.090	&	-1.094	&	-1.083	&	-1.058	\\
	&		&	(0.054)	&	(0.050)	&	(0.050)	&	(0.050)	&	(0.138)	&	(0.141)	&	(0.139)	&	(0.133)	\\
\hline																			
\multirow{2}{*}{400}	&	\multirow{2}{*}{2}	&	-1.026	&	\textbf{-1.067}	&	-1.064	&	-1.055	&	-1.161	&	\textbf{-1.166}	&	-1.154	&	-1.128	\\
	&		&	(0.139)	&	(0.137)	&	(0.136)	&	(0.134)	&	(0.167)	&	(0.170)	&	(0.169)	&	(0.161)	\\
\hline																			
\multirow{2}{*}{400}	&	\multirow{2}{*}{3}	&	-0.928	&	-0.980	&	-0.981	&	-0.982	&	-1.045	&	-1.052	&	-1.046	&	-1.029	\\
	&		&	(0.187)	&	(0.180)	&	(0.178)	&	(0.174)	&	(0.202)	&	(0.205)	&	(0.202)	&	(0.192)	\\
\hline																			
\multirow{2}{*}{400}	&	\multirow{2}{*}{4}	&	-0.384	&	-0.494	&	-0.538	&	-0.606	&	-0.506	&	-0.520	&	-0.565	&	-0.622	\\
	&		&	(0.266)	&	(0.253)	&	(0.246)	&	(0.239)	&	(0.276)	&	(0.278)	&	(0.275)	&	(0.261)	\\
\hline																			
\hline																			
\multirow{2}{*}{1225}	&	\multirow{2}{*}{0}	&	-0.922	&	-	&	-	&	-	&	-0.584	&	-	&	-	&	-	\\
	&		&	(0.021)	&	-	&	-	&	-	&	(0.032)	&	-	&	-	&	-	\\
\hline																			
\multirow{2}{*}{1225}	&	\multirow{2}{*}{1}	&	-0.878	&	-0.928	&	-0.928	&	-0.928	&	-1.089	&	-1.093	&	-1.083	&	-1.058	\\
	&		&	(0.025)	&	(0.025)	&	(0.025)	&	(0.025)	&	(0.078)	&	(0.080)	&	(0.079)	&	(0.075)	\\
\hline																			
\multirow{2}{*}{1225}	&	\multirow{2}{*}{2}	&	-1.052	&	\textbf{-1.090}	&	-1.089	&	-1.086	&	-1.189	&	\textbf{-1.194}	&	-1.180	&	-1.153	\\
	&		&	(0.074)	&	(0.074)	&	(0.074)	&	(0.074)	&	(0.094)	&	(0.096)	&	(0.094)	&	(0.089)	\\
\hline																			
\multirow{2}{*}{1225}	&	\multirow{2}{*}{3}	&	-1.043	&	-1.087	&	-1.087	&	-1.084	&	-1.163	&	-1.168	&	-1.157	&	-1.131	\\
	&		&	(0.092)	&	(0.091)	&	(0.091)	&	(0.090)	&	(0.102)	&	(0.104)	&	(0.102)	&	(0.097)	\\
\hline																			
\multirow{2}{*}{1225}	&	\multirow{2}{*}{4}	&	-0.867	&	-0.932	&	-0.942	&	-0.961	&	-0.977	&	-0.985	&	-0.988	&	-0.987	\\
	&		&	(0.114)	&	(0.112)	&	(0.110)	&	(0.107)	&	(0.120)	&	(0.122)	&	(0.119)	&	(0.113)	\\
\hline																			
\hline																			
\multirow{2}{*}{2500}	&	\multirow{2}{*}{0}	&	-0.923	&	-	&	-	&	-	&	-0.583	&	-	&	-	&	-	\\
	&		&	(0.016)	&	-	&	-	&	-	&	(0.024)	&	-	&	-	&	-	\\
\hline																			
\multirow{2}{*}{2500}	&	\multirow{2}{*}{1}	&	-0.881	&	-0.931	&	-0.931	&	-0.930	&	-1.091	&	-1.095	&	-1.084	&	-1.059	\\
	&		&	(0.019)	&	(0.018)	&	(0.018)	&	(0.018)	&	(0.059)	&	(0.060)	&	(0.059)	&	(0.056)	\\
\hline																			
\multirow{2}{*}{2500}	&	\multirow{2}{*}{2}	&	-1.063	&	-1.101	&	-1.101	&	-1.099	&	-1.196	&	-1.201	&	-1.186	&	-1.158	\\
	&		&	(0.056)	&	(0.055)	&	(0.055)	&	(0.055)	&	(0.069)	&	(0.071)	&	(0.070)	&	(0.066)	\\
\hline																			
\multirow{2}{*}{2500}	&	\multirow{2}{*}{3}	&	-1.079	&	\textbf{-1.119}	&	\textbf{-1.119}	&	-1.116	&	-1.197	&	\textbf{-1.202}	&	-1.188	&	-1.161	\\
	&		&	(0.065)	&	(0.064)	&	(0.064)	&	(0.064)	&	(0.071)	&	(0.072)	&	(0.071)	&	(0.067)	\\
\hline																			
\multirow{2}{*}{2500}	&	\multirow{2}{*}{4}	&	-1.000	&	-1.051	&	-1.056	&	-1.063	&	-1.111	&	-1.117	&	-1.108	&	-1.092	\\
	&		&	(0.073)	&	(0.072)	&	(0.072)	&	(0.071)	&	(0.078)	&	(0.080)	&	(0.078)	&	(0.074)	\\
\hline																			
\hline																			
\multirow{2}{*}{4225}	&	\multirow{2}{*}{0}	&	-0.922	&	-	&	-	&	-	&	-0.584	&	-	&	-	&	-	\\
	&		&	(0.012)	&	-	&	-	&	-	&	(0.018)	&	-	&	-	&	-	\\
\hline																			
\multirow{2}{*}{4225}	&	\multirow{2}{*}{1}	&	-0.881	&	-0.930	&	-0.930	&	-0.930	&	-1.089	&	-1.092	&	-1.081	&	-1.057	\\
	&		&	(0.014)	&	(0.013)	&	(0.013)	&	(0.013)	&	(0.045)	&	(0.046)	&	(0.045)	&	(0.043)	\\
\hline																			
\multirow{2}{*}{4225}	&	\multirow{2}{*}{2}	&	-1.063	&	-1.101	&	-1.101	&	-1.100	&	-1.197	&	-1.201	&	-1.186	&	-1.158	\\
	&		&	(0.041)	&	(0.040)	&	(0.040)	&	(0.040)	&	(0.053)	&	(0.054)	&	(0.053)	&	(0.050)	\\
\hline																			
\multirow{2}{*}{4225}	&	\multirow{2}{*}{3}	&	-1.087	&\textbf{	-1.128}	&	-1.127	&	-1.125	&	-1.207	&	\textbf{-1.212}	&	-1.196	&	-1.169	\\
	&		&	(0.047)	&	(0.047)	&	(0.047)	&	(0.046)	&	(0.054)	&	(0.055)	&	(0.054)	&	(0.051)	\\
\hline																			
\multirow{2}{*}{4225}	&	\multirow{2}{*}{4}	&	-1.044	&	-1.091	&	-1.094	&	-1.097	&	-1.158	&	-1.163	&	-1.152	&	-1.130	\\
	&		&	(0.051)	&	(0.050)	&	(0.050)	&	(0.050)	&	(0.056)	&	(0.057)	&	(0.056)	&	(0.054)	\\
\hline
\hline
\end{tabular}
\caption{Approximate validation criterion from \eqref{validationCriterion} computed for the density estimation problem with the Haar wavelet and the D4-wavelet. The hard thresholding estimator is computed w.r.t. the levels $j_0=0$ and $j_1=1,\ldots,4$. The thresholds 0.1, 0.2, 0.3 are relative thresholds as explained in the text.}
\label{TableExample2}
\end{center}
\end{table}

\begin{table}[ht]
\begin{center}
\begin{tabular}{|| c || c || c || c | c | c || c || c | c | c ||  }
\hline
\hline
	& & \multicolumn{4}{|| c ||}{ Haar } & \multicolumn{4}{| c ||}{ D4 } \\
\hline
\hline
 sample size & j & linear & \multicolumn{3}{c ||}{ nonlinear: hard threshold } & linear & \multicolumn{3}{c ||}{ nonlinear: hard threshold } \\
\hline
	& &					&	0.1  & 0.2 & 0.3 &			&	0.1 & 0.2 & 0.3 \\
\hline
\hline
\multirow{2}{*}{400}	&	\multirow{2}{*}{0}	&	-0.923	&	-	&	-	&	-	&	-0.587	&	-	&	-	&	-	\\
	&		&	(0.036)	&	-	&	-	&	-	&	(0.051)	&	-	&	-	&	-	\\
\hline																			
\multirow{2}{*}{400}	&	\multirow{2}{*}{1}	&	-0.875	&	-0.927	&	-0.926	&	-0.926	&	-1.087	&	-1.091	&	-1.082	&	-1.056	\\
	&		&	(0.049)	&	(0.046)	&	(0.046)	&	(0.046)	&	(0.130)	&	(0.133)	&	(0.132)	&	(0.125)	\\
\hline																			
\multirow{2}{*}{400}	&	\multirow{2}{*}{2}	&	-1.035	&	\textbf{-1.073}	&	-1.071	&	-1.063	&	-1.167	&	\textbf{-1.173}	&	-1.162	&	-1.134	\\
	&		&	(0.132)	&	(0.127)	&	(0.126)	&	(0.126)	&	(0.151)	&	(0.154)	&	(0.153)	&	(0.146)	\\
\hline																			
\multirow{2}{*}{400}	&	\multirow{2}{*}{3}	&	-0.933	&	-0.985	&	-0.987	&	-0.988	&	-1.048	&	-1.055	&	-1.049	&	-1.031	\\
	&		&	(0.178)	&	(0.167)	&	(0.165)	&	(0.162)	&	(0.187)	&	(0.190)	&	(0.188)	&	(0.180)	\\
\hline																			
\multirow{2}{*}{400}	&	\multirow{2}{*}{4}	&	-0.388	&	-0.497	&	-0.541	&	-0.614	&	-0.505	&	-0.519	&	-0.560	&	-0.620	\\
	&		&	(0.263)	&	(0.247)	&	(0.242)	&	(0.236)	&	(0.264)	&	(0.265)	&	(0.259)	&	(0.246)	\\
\hline																			
\hline																			
\multirow{2}{*}{1225}	&	\multirow{2}{*}{0}	&	-0.922	&	-	&	-	&	-	&	-0.584	&	-	&	-	&	-	\\
	&		&	(0.019)	&	-	&	-	&	-	&	(0.027)	&	-	&	-	&	-	\\
\hline																			
\multirow{2}{*}{1225}	&	\multirow{2}{*}{1}	&	-0.879	&	-0.929	&	-0.929	&	-0.929	&	-1.093	&	-1.097	&	-1.087	&	-1.060	\\
	&		&	(0.027)	&	(0.024)	&	(0.024)	&	(0.024)	&	(0.068)	&	(0.069)	&	(0.069)	&	(0.065)	\\
\hline																			
\multirow{2}{*}{1225}	&	\multirow{2}{*}{2}	&	-1.053	&	\textbf{-1.092}	&	-1.091	&	-1.088	&	-1.190	&	\textbf{-1.195}	&	-1.183	&	-1.153	\\
	&		&	(0.066)	&	(0.062)	&	(0.062)	&	(0.062)	&	(0.077)	&	(0.079)	&	(0.078)	&	(0.074)	\\
\hline																			
\multirow{2}{*}{1225}	&	\multirow{2}{*}{3}	&	-1.047	&	-1.090	&	-1.090	&	-1.088	&	-1.165	&	-1.170	&	-1.160	&	-1.132	\\
	&		&	(0.081)	&	(0.078)	&	(0.077)	&	(0.076)	&	(0.086)	&	(0.088)	&	(0.087)	&	(0.083)	\\
\hline																			
\multirow{2}{*}{1225}	&	\multirow{2}{*}{4}	&	-0.870	&	-0.933	&	-0.945	&	-0.965	&	-0.978	&	-0.986	&	-0.990	&	-0.987	\\
	&		&	(0.102)	&	(0.097)	&	(0.096)	&	(0.093)	&	(0.107)	&	(0.108)	&	(0.105)	&	(0.100)	\\
\hline																			
\hline																			
\multirow{2}{*}{2500}	&	\multirow{2}{*}{0}	&	-0.923	&	-	&	-	&	-	&	-0.585	&	-	&	-	&	-	\\
	&		&	(0.015)	&	-	&	-	&	-	&	(0.021)	&	-	&	-	&	-	\\
\hline																			
\multirow{2}{*}{2500}	&	\multirow{2}{*}{1}	&	-0.881	&	-0.932	&	-0.932	&	-0.931	&	-1.094	&	-1.097	&	-1.088	&	-1.061	\\
	&		&	(0.018)	&	(0.018)	&	(0.018)	&	(0.017)	&	(0.051)	&	(0.052)	&	(0.051)	&	(0.048)	\\
\hline																			
\multirow{2}{*}{2500}	&	\multirow{2}{*}{2}	&	-1.062	&	-1.101	&	-1.101	&	-1.099	&	-1.199	&	-1.203	&	-1.190	&	-1.159	\\
	&		&	(0.050)	&	(0.047)	&	(0.047)	&	(0.047)	&	(0.057)	&	(0.058)	&	(0.057)	&	(0.054)	\\
\hline																			
\multirow{2}{*}{2500}	&	\multirow{2}{*}{3}	&	-1.079	&	\textbf{-1.121}	&	-1.120	&	-1.117	&	-1.199	&	\textbf{-1.204}	&	-1.192	&	-1.162	\\
	&		&	(0.056)	&	(0.055)	&	(0.055)	&	(0.055)	&	(0.061)	&	(0.062)	&	(0.061)	&	(0.058)	\\
\hline																			
\multirow{2}{*}{2500}	&	\multirow{2}{*}{4}	&	-1.002	&	-1.052	&	-1.057	&	-1.063	&	-1.114	&	-1.120	&	-1.113	&	-1.094	\\
	&		&	(0.064)	&	(0.064)	&	(0.064)	&	(0.063)	&	(0.069)	&	(0.070)	&	(0.068)	&	(0.065)	\\
\hline																			
\hline																			
\multirow{2}{*}{4225}	&	\multirow{2}{*}{0}	&	-0.922	&	-	&	-	&	-	&	-0.586	&	-	&	-	&	-	\\
	&		&	(0.011)	&	-	&	-	&	-	&	(0.015)	&	-	&	-	&	-	\\
\hline																			
\multirow{2}{*}{4225}	&	\multirow{2}{*}{1}	&	-0.882	&	-0.932	&	-0.932	&	-0.932	&	-1.095	&	-1.098	&	-1.089	&	-1.062	\\
	&		&	(0.014)	&	(0.013)	&	(0.013)	&	(0.013)	&	(0.038)	&	(0.039)	&	(0.039)	&	(0.036)	\\
\hline																			
\multirow{2}{*}{4225}	&	\multirow{2}{*}{2}	&	-1.067	&	-1.104	&	-1.104	&	-1.104	&	-1.202	&	-1.207	&	-1.193	&	-1.162	\\
	&		&	(0.035)	&	(0.033)	&	(0.033)	&	(0.033)	&	(0.042)	&	(0.043)	&	(0.043)	&	(0.040)	\\
\hline																			
\multirow{2}{*}{4225}	&	\multirow{2}{*}{3}	&	-1.093	&	\textbf{-1.132}	&	\textbf{-1.132}	&	-1.130	&	-1.212	&	\textbf{-1.217}	&	-1.204	&	-1.174	\\
	&		&	(0.039)	&	(0.038)	&	(0.038)	&	(0.038)	&	(0.043)	&	(0.044)	&	(0.044)	&	(0.041)	\\
\hline																			
\multirow{2}{*}{4225}	&	\multirow{2}{*}{4}	&	-1.049	&	-1.095	&	-1.097	&	-1.101	&	-1.163	&	-1.169	&	-1.159	&	-1.134	\\
	&		&	(0.044)	&	(0.043)	&	(0.042)	&	(0.041)	&	(0.047)	&	(0.048)	&	(0.048)	&	(0.045)	\\
\hline
\hline
\end{tabular}
\caption{Approximate validation criterion from Equation \eqref{validationCriterion} with independent reference samples. The hard thresholding estimator is computed w.r.t. the levels $j_0=0$ and $j_1=1,\ldots,4$. The thresholds 0.1, 0.2, 0.3 are relative thresholds as explained in the text.}
\label{TableExample2_ind_ref_sample}
\end{center}
\end{table}

\section{Proofs of the theorems in Section \ref{WaveletResults}} \label{Proofs_of_the_main_theorems}
Throughout this section, we use the common convention to abbreviate constants in $\R$ by $A_i$ or $A$ or likewise by $C_i$ or $C$. Furthermore, we use the convention to write $\norm{\,\cdot\,}_p$ for the norm of $L^p(\lambda^d)$, $p\in[1,\infty]$. The idea of the first lemma dates back at least to \cite{meyer1990ondelettes}. It applies in particular to wavelets $\Psi_k$ which have compact support.
\begin{lemma}[Norm equivalence on Besov spaces]\label{BesovEquivalence}
The norms in \eqref{GeneralizedBesovNorm} and in \eqref{GeneralizedBesovNorm2} are equivalent provided that the wavelets $\Psi_k$ are integrable and $\sup_{x\in\R^d} \sum_{\gamma\in\Z^d} |\Psi_k(x-\gamma)| < \infty$ for each $k=0,\ldots,|M|-1$.
\end{lemma}
\begin{proof}
We show that there are $0<C_1,C_2<\infty$ depending on $s,p,q$ such that $C_1 \norm{f}_{s,p,q} \le \norm{f}_{B^s_{p,q}} \le C_2 \norm{f}_{s,p,q}$. First we consider the left inequality: define for $j\ge j_0$ the functions $g^{(k)}_j \coloneqq \sum_{\gamma \in \Z^d} \upsilon_{k,j,\gamma} \, \Psi_{k,j,\gamma}$ for $k=1,\ldots,|M|-1$ and $g^{(0)}_j \coloneqq \sum_{\gamma\in\Z^d} \theta_{j_0,\gamma} \, \Phi_{j_0,\gamma}$.  Denote the H{\"o}lder conjugate of $p$ by $u$, then by the property of an orthonormal basis and H{\"o}lder's inequality applied to the measure $|\Psi_{k,j,\gamma}| \,\intd{\lambda^d}$
\begin{align*}
		|\upsilon_{k,j,\gamma}| &\le \left(	\int_{\R^d} |g^{(k)}_j|^p \, |\Psi_{k,j,\gamma}|	\,\intd{\lambda^d} \right)^{1/p}  \left(	\int_{\R^d} |\Psi_{k,j,\gamma}|	\,\intd{\lambda^d}	\right)^{1/u}, \\
		&\text{ thus, } \norm{ \upsilon_{k,j,\cdot} }_{\ell^p} \le |M|^{j(1/p-1/2)} \, \norm{\Psi_k}_1^{1/u} \,\norm{g^{(k)}_j}_p \norm{ \sum_{\gamma\in\Z^d} |\Psi_k(\cdot-\gamma)| }_{\infty}^{1/p}
\end{align*}
with the usual modification if $p=1$ or $p=\infty$; the same reasoning is true for the vector $\theta_{j_0,\cdot}$. Then,
\begin{align*}
	\norm{f}_{B^s_{p,q}} \ge C_1 \norm{f}_{s,p,q}, \text{ where } C_1 \coloneqq \min_{0\le k \le |M|-1} \left\{	\norm{\Psi_k}_1^{-1/u} \, \norm{ \sum_{\gamma\in\Z^d} |\Psi_k(\cdot-\gamma)| }_{\infty}^{-1/p} 	\right\} < \infty.
\end{align*}
For the right inequality, consider the following pointwise inequality
\begin{align*}
	|g^{(k)}_j| \le \sum_{\gamma\in\Z^d} \left| \upsilon_{k,j,\gamma} \right|\, \left|\Psi_{k,j,\gamma}		\right|^{1/p} \, \left|\Psi_{k,j,\gamma}		\right|^{1/u} \le \left( \sum_{\gamma\in\Z^d} |\upsilon_{k,j,\gamma}|^p \, |\Psi_{k,j,\gamma} | \right)^{1/p} \left( \sum_{\gamma\in\Z^d} |\Psi_{k,\ell,\gamma}| \right)^{1/u}
	\end{align*}
for $k=1,\ldots,|M|-1$; it is also true if $k=0$. Thus,
\begin{align*}
		\norm{ g^{(k)}_j }_p \le \norm{ \sum_{\gamma\in\Z^d} |\Psi_k(\,\cdot\,-\gamma)| }_{\infty}^{1/u} \norm{ \Psi_k }_1^{1/p} |M|^{j(1/2-1/p)} \norm{\upsilon_{k,j,\cdot} }_{\ell^p}. 
\end{align*}
Hence, $\norm{f}_{B^s_{p,q}} \le C_2 \norm{f}_{s,p,q}$ with $C_2 \coloneqq \max_{0\le k \le |M|-1} \norm{ \sum_{\gamma\in\Z^d} | \Psi_k(\,\cdot\, - \gamma ) | }_{\infty}^{1/u} \norm{\Psi_k}_1^{1/p} < \infty$.
\end{proof}

We are now prepared to give bounds on the estimation error of the linear estimator.
\begin{theorem}\label{LpConvergence}
Let $p'\in [1,\infty)$ and assume the density $f$ to be in $L^{p'}(\lambda^d)\cap L^\infty(\lambda^d)$.
\begin{enumerate}

	\item If $p'\in [1,2]$ and if Condition~\ref{regCond0}~\ref{DensityRequirement} is satisfied with $a=1$, then
	\begin{align*}
				\E{ \int_{\R^d} \left|\tilde{P}_j f - P_j f \right|^{p'} \,\intd{\lambda^d} } ^{1/p'} &\le C_{p'}\, (2L+1)^{d} \, \norm{\Phi}_{L^\infty(\lambda^d)} \norm{\Phi}_{L^2(\lambda^d)}\\
				&\quad \cdot \left\{  \norm{h^{1/2}}_{L^1(\lambda^d)}^{1/2} + \norm{\tilde{h}}_{L^2(\lambda^d)}	\right\} \frac{ |M|^{j/2}}{|I_n|^{1/2} }
	\end{align*}

	\item If $p'\in (2,\infty)$ and Condition~\ref{regCond0}~\ref{DensityRequirement} is satisfied with $a=2$, then
	\begin{align*}
	\E{ \int_{\R^d} |\tilde{P}_j f - P_j f|^{p'}\, \intd{\lambda^d} }^{1/p'}  &\le C_{p'}\, (2L+1)^{d} \, \norm{\Phi}_{L^\infty(\lambda^d)}^{1/p'} \norm{\Phi}_{L^{p'}(\lambda^d)}  \Bigg\{ \norm{h^{1/4}}_{L^1(\lambda^d)}^{1/p'} + \norm{\tilde{h}}_{L^1(\lambda^d)}^{1/p'}	\Bigg\}\\
		&\quad \cdot  		 \left\{ \frac{ |M|^{j(1-1/(2p'))} q_n^{N(1-1/p') } }{|I_n|^{1-1/(2p')} }   + \frac{ |M|^{j/2} }{|I_n|^{1/2} } \right\},
\end{align*}
where $q_n = B \ln |I_n|$ for $B>0$ arbitrary but fixed.
\end{enumerate}

The constant $C_{p'}$ depends on $p'$, the bound of the mixing coefficients which is given by the numbers $c_0, c_1 \in \R_+$. If $p'>2$, $C_{p'}$ depends additionally on $B$.
\end{theorem}

\begin{proof}[Proof of Theorem~\ref{LpConvergence}]
We write $\tilde{f}_j$ (resp. $f_j$) instead of $\tilde{P}_j f$ (resp. $P_j f$) to keep the notation simple. W.l.o.g.\ the support of $\Phi$ is contained in $[0,L]^d$, $L\in\N_+$. Hence, a fixed $x\in\R^d$ is at most contained in the support of $(2L+1)^d$ translations of $\Phi$. Consequently, if we apply the H{\"o}lder inequality to the counting measure over the index $\gamma$, the estimation error is at most
\begin{align}
		\E{ \int_{\R^d} |f_j - \tilde{f}_j |^{p'} \,\intd{\lambda^d} } \le (2L+1)^{d(p'-1)} \norm{\Phi}_{p'}^{p'} |M|^{j(p'/2-1)} \sum_{\gamma \in \Z^d} \E{ |\hat{\theta}_{j,\gamma} - \theta_{j,\gamma} |^{p'} }.  \label{BoundEstimationErrorV1}
\end{align}
We investigate the sum in \eqref{BoundEstimationErrorV1}. Firstly, let ${p'}\in [1,2]$, then we obtain with Proposition~\ref{IntegrabilityForDependentSums} that
$$
		\left( \sum_{\gamma \in \Z^d} \E{ |\hat{\theta}_{j,\gamma} - \theta_{j,\gamma} |^2 } \right)^{1/2} \le C_2 L^{d/2} 2^{d/2 }\norm{\Phi}_{\infty} \left\{ \norm{h^{1/2}}_1^{1/2} +  \norm{\tilde{h} }_2 \right\} \frac{|M|^{j/2}}{|I_n|^{1/2} }.
$$
This yields the claim for $p'\le 2$, if we use that $(2L+1)^{d/2} L^{d/2} 2^{d/2} \le (2L+1)^d$ and the H{\"o}lder inequality to bound $\E{\int_{\R^d} |f_j - \tilde{f}_j |^{p'} \,\intd{\lambda^d}}^{1/p'}$ by $\E{\int_{\R^d} |f_j - \tilde{f}_j |^{2} \,\intd{\lambda^d}}^{1/2}$. 

Secondly, if $p' > 2$, we use the decomposition
\begin{align}\label{BoundEstimationErrorV2}
				\sum_{\gamma\in\Z^N} \E{ |\hat{\theta}_{j,\gamma} - \theta_{j,\gamma} |^{p'} } \le \sum_{\gamma\in\Z^N} \E{ |\hat{\theta}_{j,\gamma} - \theta_{j,\gamma} |^{2} }^{1/2} \E{ |\hat{\theta}_{j,\gamma} - \theta_{j,\gamma} |^{2(p'-1)} }^{1/2}. 
\end{align}
We bound the last factor inside the sum with Lemma~\ref{LemmaMomentInequalityWavelets}, note that $2(p'-1)>2$:
\begin{align*}
		\E{ |\hat{\theta}_{j,\gamma} - \theta_{j,\gamma} |^{2(p'-1)} }^{1/2} =  C_{2(p'-1) } \left\{ |I_n|^{-(p'-1) } + \left(\frac{|M|^{j/2} q_n^N}{|I_n|}\right)^{2(p'-1) } + \frac{ |M|^{j(p'-1) } }{|I_n|^{c_1 B} }  \right\}^{1/2},
		\end{align*}
where the constant $C_{2(p'-1) }$ depends on $B$. If we choose $B>0$ sufficiently large, then the last term inside the parentheses is negligible. Moreover, it follows with Proposition~\ref{IntegrabilityForDependentSums} that 
$$
	\sum_{\gamma\in\Z^N} \E{ |\hat{\theta}_{j,\gamma} - \theta_{j,\gamma} |^{2} }^{1/2} \le C_2 L^{d/2} 2^{d} \norm{\Phi}_\infty \left\{ \norm{h^{1/2}}_1 +  \norm{h^{1/4}}_1 +  \norm{\tilde{h} }_1 \right\} |M|^{j} {|I_n|^{-1/2} }.
$$	
We combine \eqref{BoundEstimationErrorV1} and \eqref{BoundEstimationErrorV2} to obtain the result. Note that we have $(2L+1)^{d(p'-1)/p'} L^{d/2} 2^{d/p'} \le (2L+1)^d$
\end{proof}

It follows the proof of Theorem~\ref{LinearDensityEstimationBesovFunction} which quantifies the rate of convergence of the linear estimator
\begin{proof}[Proof of Theorem~\ref{LinearDensityEstimationBesovFunction}]
Denote the H{\"o}lder conjugate of $p'$ by $u$, i.e., $(p')^{-1} + u^{-1} = 1$. We show that the approximation error $\norm{f - P_j f }_{p'}$ can be bounded by
\begin{align}\begin{split}\label{ApproximationErrorEq}
 \norm{f - P_j f }_{p'} &\le C_A \max_{1  \le k \le  |M|-1 } \norm{\Psi_k }_{1}  \max_{1 \le k \le |M|-1} \norm{ \sum_{\gamma\in\Z^d} |\Psi_k(\cdot-\gamma)| }_{\infty}^{1/u} \\
&\quad  \cdot \norm{f}_{s,p, \infty} |M|^{1-j s'} / (1-|M|^{-s'} ), 
\end{split}\end{align}
where the constant $C_A$ only differs from 1 if $p > p'$, in this case it depends on the domain $A$. 

We have to distinguish the cases $p \le p'$ and $p>p'$ but can treat this in one formula. We proceed as in the proof of Lemma~\ref{BesovEquivalence}:
\[
		\norm{ \sum_{\gamma\in\Z^d} \upsilon_{k,j,\gamma} \,\Psi_{k,j,\gamma} }_{p'} \le \max_{1 \le k \le |M|-1} \norm{ \sum_{\gamma\in\Z^d} |\Psi_k(\cdot-\gamma)| }_{\infty}^{1/u} \, \norm{ \Psi_k }_1^{1/p'} \,|M|^{j(1/2-1/p')} \,\norm{\upsilon_{k,j,\cdot}}_{\ell^{p'}},
\]
with the notation that $u$ is the H{\"o}lder conjugate to $p'$. In the case $p > p'$, the number of nonzero coefficients on the $j$-th level (for the $k$-th mother wavelet) is bounded by $C_A |M|^{j} $, where $C_A$ depends on the domain of $f$. This follows from the dilatation rules of volumes under linear transformations and from the fact that the domain $A$ is bounded. Consequently, we have in both cases $p\ge p'$ and $p< p'$ the inequalities for the $\ell^p$-sequence norms,
\[
	\norm{\upsilon_{k,j,\cdot}}_{\ell^{p'}} \le C_A \, |M|^{j(1/p' - 1/p)^+} \norm{ \upsilon_{k,j,\cdot}}_{\ell^{p}}
\]
where $C_A=1$ if $p' \ge p$. Then with H{\"o}lder's inequality and the Besov property of $f$,
\begin{align}\begin{split}
 \norm{f - P_j f }_{p'} &\le C_A \max_{1  \le k \le  |M|-1 } \norm{\Psi_k }_1^{1/p'}  \max_{1 \le k \le |M|-1} \norm{ \sum_{\gamma\in\Z^d} |\Psi_k(\cdot-\gamma)| }_{\infty}^{1/u}\, \\
&\quad\qquad\qquad\qquad \cdot \norm{f}_{s,p, \infty} |M|^{1-j s'} / (1-|M|^{-s'} ) \le C |M|^{-j s'} \label{BesovApproxError}
\end{split} \end{align}
with the definition $s' = s+(1/p' - 1/p) \wedge 0$. Note that $s' > 0$ as $s > 1/p$. The constant $C$ depends on the matrix $M$, the wavelets, $f$ and if $p<p'$ additionally on the domain $A$. The estimation error is given in Theorem~\ref{LpConvergence}. The growth rate of $j$ equalizes the rates of the terms $|M|^{-js'}$ and $|M|^{j/2} |I_n|^{1/2}$; both behave as $|I_n|^{-s'/(2s'+1)}$. This implies in particular that the term $ |M|^{j(1-1/(2p'))} q_n^{N(1-1/p') } \big/ |I_n|^{1-1/(2p')}  $, which appears in the case $p'>2$, is negligible. This proves the first statement of this theorem.

The amendment concerning the rate of convergence of the supremum $\sup_{f\in F_{s,p,q}(K,A)}  \E{ \int_{\R^d} |f - \tilde{P}_j f |^{p'}}^{1/p'}$ can be easily verified now. Since the support of $f$ is contained in a bounded set, the integrability requirement for the dominating function $h$ is satisfied because the requirement $\norm{f}_{s,p,q}\le K$ implies a uniform bound on the maximum norm of $f$. We only need that the mutual dependence between the variables $Z(s)$ and $Z(t)$ is as required in Condition~\ref{regCond0}~\ref{DensityRequirement}. This however follows by the assumptions of the amendment.
\end{proof} 

\begin{proof}[Proof of Theorem~\ref{DifferentiableDensity}]
We prove that the approximation error is in $\cO\left(		(\zeta_{min})^{-j} \right)$; the claim follows then with an application of Theorem~\ref{LpConvergence}. Since the father and mother wavelets $\Psi_k$ are compactly supported on $[0,L]^d$, there are at most $(2L+1)^d$ wavelets not equal to zero for fix $x\in\R^d$. Hence, for all $j \in \Z$ and $k\in \{1,\ldots,|M|-1\}$
\begin{align*}
		\int_{\R^d} \bigg| \sum_{\gamma\in \Z^d} \upsilon_{k,j,\gamma}\, \Psi_{k,j,\gamma} \bigg|^{p'} \intd{\lambda^d} \le (2L+1)^{d p'} \norm{\Psi_k}_{p'}^{p'} |M|^{j({p'}/2-1)} \sum_{\gamma \in \Z^d} \left|\upsilon_{k,j,\gamma} \right|^{p'} = \cO\left((\zeta_{min})^{-j {p'}}		\right).
\end{align*} 
Here we use the following bound on the wavelet coefficients $\upsilon_{k,\ell,\gamma}$
\begin{align}
		|\upsilon_{k,j,\gamma}|^{p'} &\le |M|^{-jp/2}\, \norm{\Psi_k}_1^{p'}\, \sup \left\{ |f(x)-f(y)| : x,y \in \supp{\Psi_{k,j,\gamma}} \right \} ^{p'} \nonumber \\
		&\le |M|^{-j {p'}/2}\, \norm{\Psi_k}_1^{p'}\, \left[ \sup\left\{ \bar{h} \left(M^{-j}(u+\gamma) \right): u \in [0,L]^d \right\} \, \norm{M^{-j} }_2\, \sqrt{d} L\right]^{p'}. \nonumber
\end{align}
Thus, the approximation error is bounded by	$\norm{ f-P_j f }_{p'}  \le \sum_{k=1}^{|M|-1} \sum_{\ell =j}^{\infty}\norm{ \sum_{\gamma \in \Z^d} \upsilon_{k,\ell,\gamma} \, \Psi_{k,\ell,\gamma} }_{p'} = \cO\left(	(\zeta_{min})^{-j} \right)$.
\end{proof}

We prove now the statement concerning the rate of convergence of the hard thresholding density estimator.
\begin{proof}[Proof of Theorem \ref{LPconvergenceHardThresholding}]
Write the approximation w.r.t.\ to the $j_1$-th and $j_0$-th level as
\begin{align*}
		Q_{j_0,j_1} f &= P_{j_1} f = \sum_{\gamma\in\Z^d} \theta_{j_0,\gamma} \Phi_{j_0,\gamma} + \sum_{k=1}^{|M|-1} \sum_{j = j_0}^{j_1 - 1} \upsilon_{k,j,\gamma} \Psi_{k,j,\gamma}.
\end{align*}
We decompose the error as follows
\begin{align}
	\E{ \norm{ f - \tilde{Q}_{j_0,j_1} f }_{p'}^{p'} }^{1/p'} & \le \norm{ f - Q_{j_0,j_1} f }_{p'}  +  \E{ \norm{ \sum_{\gamma\in\Z^d} (\hat{\theta}_{j_0,\gamma} - \theta_{j_0,\gamma} )\, \Phi_{j_0,\gamma} }_{p'}^{p'}}^{1/p' } \nonumber \\
	&\quad  + \sum_{k=1}^{|M|-1} \sum_{j=j_0}^{j_1-1} \E{ \norm{  \sum_{\gamma\in\Z^d}  \left(\hat{\upsilon }_{k,j,\gamma} \1{  |\hat{\upsilon }_{k,j,\gamma}| > \bar{\lambda}_{j}  } - \upsilon_{k,j,\gamma} \right)  \, \Psi_{k,j,\gamma} 	}_{p'}^{p'}}^{1/p' }  \nonumber \\
	&=: J_1 + J_2 + J_3  \label{HardThresholdingEq1}
\end{align}
and consider these three terms separately. We infer from Theorem~\ref{LinearDensityEstimationBesovFunction} that the approximation error $J_1$ is at most
$
	|M|^{-j_1 s'} \simeq |I_n|^{-\alpha}
$
times a constant which only depends on the domain $A$, the parameters of the Besov space, the wavelets $\Psi_k$ and the Besov norm $\norm{f}_{s,p,\infty}\le\norm{f}_{s,p,q}$; for its exact value see \eqref{BesovApproxError}. 

For linear estimation error $J_2$, we use Theorem~\ref{LpConvergence}. So, $J_2 \le C |M|^{j_0/2}/|I_n|^{1/2} \simeq |I_n|^{-1/2}$. We consider the nonlinear details term in the estimation error which is the third term on the RHS of \eqref{HardThresholdingEq1} and which constitutes the main error. It can be decomposed and bounded as follows
\begin{align}
		J_3 &\le  (2L+1)^{d(p'-1)/p'}  \sum_{k=1}^{|M|-1} \sum_{j = j_0}^{j_1 - 1} |M|^{j(1/2 - 1/p')} \norm{ \Psi_k}_{p'} \nonumber \\ 
		\begin{split} \label{HardThresholdingEq2}
		&\quad\quad  \cdot \left\{ 	 \left(\sum_{\gamma\in\Z^d}  |\upsilon_{k,j,\gamma} |^{p'}\, \I\{ |\upsilon_{k,j,\gamma} | \le 2\bar{\lambda}_j \}  + \left(\sum_{\gamma\in\Z^d} \p\left( | \hat{\upsilon}_{k,j,\gamma} - \upsilon_{k,j,\gamma} | > \bar{\lambda}_j \right) | \upsilon_{k,j,\gamma} |^{p'}\right)^{1/p'} \right)^{1/p'} \right.  \\		
			&\quad\qquad\qquad 	+	 \left(\sum_{\gamma\in\Z^d} \E{ | \hat{\upsilon}_{k,j,\gamma} - \upsilon_{k,j,\gamma} |^{p'}\, \I\{ |\hat{\upsilon}_{k,j,\gamma} - \upsilon_{k,j,\gamma} | > \bar{\lambda}_j/2 \} }	\right)^{1/p'}  \\				
			&\quad \qquad\qquad+ \left. \left(\sum_{\gamma\in\Z^d} \E{| \hat{\upsilon}_{k,j,\gamma} - \upsilon_{k,j,\gamma} |^{p'}\, \I\{ |\upsilon_{k,j,\gamma} | > \bar{\lambda}_j/2 \}}	\right)^{1/p'} \right \}.
				\end{split}
\end{align}
We derive the rates of convergence for each term in \eqref{HardThresholdingEq2} separately, many techniques are quite similar to the classical proof given by \cite{donoho1996}.
The first error in \eqref{HardThresholdingEq2} is the dominating error. If $p' > p$, it can be bounded as
\begin{align}
		& \sum_{k=1}^{|M|-1} \sum_{j=j_0}^{j_1-1} |M|^{j(1/2-1/p')} \left(\sum_{\gamma\in\Z^d} |\upsilon_{k,j,\gamma}|^{p} \, (2\bar{\lambda}_j)^{p'-p} \, 1\left\{ |\upsilon_{k,j,\gamma}| \le 2\bar{\lambda}_j \right\}		\right)^{1/p'} \nonumber \\
		&\le \sum_{k=1}^{|M|-1} \sum_{j=j_0}^{j_1-1} |M|^{j(1/2-1/p')} \, (2\bar{\lambda}_j)^{(p'-p)/p'} \, |M|^{-j(s+1/2-1/p)p/p'} \,  \norm{f}_{s,p,\infty}^{p/p'} \nonumber \\
		&\le C \norm{f}_{s,p,\infty}^{p/p'} |I_n|^{-(p'-p)/(2p') } \sum_{k=1}^{|M|-1} \sum_{j=j_0}^{j_1 - 1}  j^{(p'-p)/(2p') } |M|^{-j \epsilon / p'}. \label{HardThresholdingEq3} 
\end{align}
If $\epsilon\neq 0$, \eqref{HardThresholdingEq3} is bounded by
\begin{align*}
		 |I_n|^{-(p'-p)/(2p') } |M|^{\max( -j_0 \epsilon/p', -j_1 \epsilon/p')} j_1^{(p'-p)/(2p')} \simeq |I_n|^{-\alpha} \left( \ln |I_n|\right)^{(p'-p)/2p'}. 
\end{align*}
If $\epsilon=0$, \eqref{HardThresholdingEq3} is bounded by $|I_n|^{-(p'-p)/(2p')} (j_1-j_0) j_1^{(p'-p)/(2p')} \simeq |I_n|^{-\alpha} \left( \ln |I_n|\right)^{(3p'-p)/2p'}$.

We treat the first error term in \eqref{HardThresholdingEq2} in the case $p\ge p'$. We have $\epsilon>0$ and $s=s'$. Moreover, the density has bounded support. We find in this case
\begin{align}
		\sum_{k=1}^{|M|-1} \sum_{j=j_0}^{j_1 - 1} |M|^{j(1/2-1/p')} \norm{ \upsilon_{k,j,\cdot} }_{\ell^{p'} } \le C_A \norm{f}_{s,p,\infty} \sum_{k=1}^{|M|-1} \sum_{j=j_0}^{j_1 - 1} |M|^{-js}, \label{HardThresholdingEq4}
\end{align}
where $C_A$ is the constant which depends on the support of $f$ and which is introduced in the proof of Theoerem~\ref{LinearDensityEstimationBesovFunction}. Consequently, this last inequality behaves as $|M|^{-j_0 s} = |I_n|^{-s/(1+2s)} = |I_n|^{-\alpha}$.

For the remaining three errors from \eqref{HardThresholdingEq2}, we need two bounds. Firstly, we prove that given the growth rate of $j_1$
\begin{align}\label{HardThresholdingEq5}
\sup_{\gamma\in \Z^d } \E{ | \hat{\upsilon}_{k,j,\gamma} - \upsilon_{k,j,\gamma} |^{p'} }^{1/p'}\le \sup_{\gamma\in \Z^d } \E{ | \hat{\upsilon}_{k,j,\gamma} - \upsilon_{k,j,\gamma} |^{2p'} }^{1/2p'} \le C |I_n|^{-1/2}.
\end{align}
We know from Lemma~\ref{LemmaMomentInequalityWavelets} that the leftmost expectation of \eqref{HardThresholdingEq5} is bounded by
$$
		C \left( |I_n|^{-1/2} + \frac{|M|^{j_1/2} q_n^N}{|I_n|} + \frac{ |M|^{j_1/2} }{|I_n|^{c_1B/(2p')}} \right)
$$
where $q_n = B \ln |I_n|$ for a $B > 0$ arbitrary but fixed and where the constant $C$ depends on $B$. Hence, if we choose $B$ sufficiently large, namely,
$$
		B\coloneqq \frac{2p' +1 + \alpha/(2s') }{c_1},
$$
it only remains to show that the term in the middle is negligible. It corresponds to $|I_n|^{\alpha/(2s')-1} (\ln |I_n|)^N$. 

Consequently, we only need that $1-\alpha/s' >0$. If $\epsilon<0$, then $1-\alpha/s' = 2 (sp-1)/(p+2(sp-1)) > 0$ because by assumption $sp>1$. Moreover, if $\epsilon \ge 0$, we use that $1/p' \ge 1/(2sp +p)$. The relation $1-\alpha/s' > 0 $ is equivalent to $s'(1+2s) > s$. Now $s'\ge s + 1/(2sp +p)-1/p$. Thus,	$s' (1+2s) \ge s ( 2s+1-2/p) > s$, where we use again that $s>1/p$. All in all, we find that \eqref{HardThresholdingEq5} is true.

Secondly, we show that $\p(|\hat{\upsilon}_{k,j,\gamma}-\upsilon_{k,j,\gamma}|>\bar{\lambda}_j )^{1/(2p')}$ vanishes at a rate $|I_n|^{-C_{p'} K_0^2}$ which is negligible given the choice of $K_0$, which is defined below in \eqref{hardThresChoiceK0}. We infer from Lemma~\ref{LemmaExpInequalityWavelets} that this probability can be bounded by
$$
		2\exp\left( - \frac{|I_n| \bar{\lambda}_j^2 / (2p') }{A_2 + A_1 q_n^N |M|^{j/2} \bar{\lambda}_j} \right) + A_3 \left( \frac{|M|^{j/2}}{\bar{\lambda}_j |I_n|^{c_1 B - 1/2} } \right)^{1/(2p')}
$$
for certain constants $A_1,\ldots,A_3$ independent of $I_n$ and $j$. So it remains to compute the asymptotics of the following three terms:
\begin{align}\begin{split}\label{HardThresholdingEq6}
		\frac{|I_n| \bar{\lambda}_j^2  }{2p' A_2} &\ge \frac{K_0^2 j_0}{2p' A_2} \simeq \frac{K_0^2 (1-2\alpha)}{ (2p'  \ln |M| ) } \ln |I_n| \\
		\frac{|I_n| \bar{\lambda}_j^2 }{q_n^N |M|^{j/2} \bar{\lambda}_j } &\ge C K_0 |I_n|^{(1-\alpha/s')/2} \ln|I_n|^{1/2-N} \\
		\frac{|M|^{j/(4p')}}{\bar{\lambda}_j^{1/(2p')} |I_n|^{(c_1 B - 1/2)/(2p')} } &\le \left(K_0^{-1} j_1^{-1/2} |I_n|^{-(c_1 B - 1 - \alpha/(2s') )} \right)^{1/(2p')}
\end{split}\end{align}
We see that the error on the second line of \eqref{HardThresholdingEq6} is negligible because $1-\alpha/s'>0$. The error on the third line vanishes at a rate greater than $|I_n|^{-1}$ and is negligible as well. Hence, the choice 
\begin{align}\label{hardThresChoiceK0}
		K_0^2 \simeq p' \ln |M| / (1-2\alpha)
\end{align}
implies that the probability in question decays at a rate of at least $|I_n|^{-1}$, i.e.,
\begin{align}\label{HardThresholdingEq7}
				\p(|\hat{\upsilon}_{k,j,\gamma}-\upsilon_{k,j,\gamma}|>\bar{\lambda}_j )^{1/(2p')} \le C |I_n|^{-1}.
\end{align}
We use the norm inequalities in the $\ell^{p'}$-spaces in both cases $p' \ge p$ and $p' < p$ to bound the second error in \eqref{HardThresholdingEq2}:
\begin{align}
	&\sum_{k=1}^{|M|-1}  \sum_{j = j_0}^{j_1 - 1} |M|^{j(1/2 - 1/p')} \left(\sum_{\gamma\in\Z^d} \p\left( | \hat{\upsilon}_{k,j,\gamma} - \upsilon_{k,j,\gamma} | > \bar{\lambda}_j \right) | \upsilon_{k,j,\gamma} |^{p'}\right)^{1/p'} \nonumber \\
	&\le C C_A \sum_{k=1}^{|M|-1}  \sum_{j = j_0}^{j_1 - 1} |M|^{j(1/2 - 1/p')} \, |M|^{j(1/p' - 1/p)^+} \norm{ \upsilon_{k,j,\cdot} }_{\ell^{p}} \, \p\left( | \hat{\upsilon}_{k,j,\gamma} - \upsilon_{k,j,\gamma} | > \bar{\lambda}_j \right)^{1/p'} \nonumber \\
	&\le C C_A \norm{f}_{s,p,\infty} \,  \sum_{k=1}^{|M|-1}  \sum_{j = j_0}^{j_1 - 1} |M|^{-js'} \p\left( | \hat{\upsilon}_{k,j,\gamma} - \upsilon_{k,j,\gamma} | > \bar{\lambda}_j \right)^{1/p'}. \nonumber
\end{align}
Consequently, this second error is negligible if we use the result from \eqref{HardThresholdingEq7}.

For the third error in \eqref{HardThresholdingEq2} we need the estimate from Proposition~\ref{IntegrabilityForDependentSums}
\begin{align*}
		\sum_{\gamma\in\Z^d}	\E{ | \hat{\upsilon}_{k,j,\gamma} - \upsilon_{k,j,\gamma} |^{2p'}}^{1/2} &\le \sum_{\gamma\in\Z^d}	\E{ | \hat{\upsilon}_{k,j,\gamma} - \upsilon_{k,j,\gamma} |^{2}}^{1/4} \E{ | \hat{\upsilon}_{k,j,\gamma} - \upsilon_{k,j,\gamma} |^{2(2p'-1)}}^{1/4} \\
		&\le C_{p'} \left\{ \left( \norm{h^{1/8} }_1 + \norm{\tilde{h}^{1/2} }_1 \right) \norm{\Psi_k}_\infty^{1/2} |M|^j \, |I_n|^{-1/4}  \right\} |I_n|^{-(2p'-1)/4}.
\end{align*}
Now we obtain, using H{\"o}lders inequality in both cases $p' \ge p$ and $p' < p$
\begin{align*}
		&\sum_{k=1}^{|M|-1}  \sum_{j = j_0}^{j_1 - 1} |M|^{j(1/2 - 1/p')} \left(\sum_{\gamma\in\Z^d}	\E{ | \hat{\upsilon}_{k,j,\gamma} - \upsilon_{k,j,\gamma} |^{2p'}}^{1/2} \p\left( | \hat{\upsilon}_{k,j,\gamma} - \upsilon_{k,j,\gamma} | > \bar{\lambda}_j / 2 \right)^{1/2} 	\right)^{1/p' }  \\
		&\le C (j_1-j_0) |I_n|^{-(1-\alpha/s') - 1/2} .
\end{align*}
Consequently, this error is negligible.

The fourth error in \eqref{HardThresholdingEq2} can be treated similar. We use that $\E{| \hat{\upsilon}_{k,j,\gamma} - \upsilon_{k,j,\gamma} |^{p'} }^{1/p'} \le C |I_n|^{-1/2}$ from Equation~\eqref{HardThresholdingEq5}. If $p' > p$, 
\begin{align}
		&\sum_{k=1}^{|M|-1}  \sum_{j = j_0}^{j_1 - 1} |M|^{j(1/2 - 1/p')} \left(\sum_{\gamma\in\Z^d} \E{| \hat{\upsilon}_{k,j,\gamma} - \upsilon_{k,j,\gamma} |^{p'} } \, \1{ |\upsilon_{k,j,\gamma} | > \bar{\lambda}_j/2 }	\right)^{1/p'} \nonumber \\
		&\le C \sum_{k=1}^{|M|-1}  \sum_{j = j_0}^{j_1 - 1} |M|^{j(1/2 - 1/p')} \, |I_n|^{-1/2}  \, \norm{ \upsilon_{k,j,\cdot}}_{\ell^p}^{p/p'} (\bar{\lambda}_j/2)^{-p/p'} \nonumber \\
		&\le C \norm{f}_{s,p,\infty}^{p/p'}\, |I_n|^{-(p'-p)/(2p')} \sum_{k=1}^{|M|-1} \sum_{j=j_0}^{j_1-1} |M|^{-j\epsilon/p'} j^{-p/(2p')}. \label{HardThresholdingEq9}
\end{align}
Note that \eqref{HardThresholdingEq9} is asymptotically less than the first nonlinear details term given in \eqref{HardThresholdingEq3} and can be neglected. In the case $p'\le p$, this error term can be bounded by 
$$
		\norm{f}_{s,p,\infty} \sum_{k=1}^{|M|-1} \sum_{j=j_0}^{j_1-1} |M|^{-j s} j^{-1/2} 
$$
times a constant. This follows similarly as the derivation of \eqref{HardThresholdingEq9}. In particular, this error is negligible too, when compared to the first details term in the case $p'\le p$, see \eqref{HardThresholdingEq4}. 

The amendment concerning the uniform convergence follows along the same lines as in the case for the linear estimator, see the proof of Theorem~\ref{LinearDensityEstimationBesovFunction}. This finishes the proof.
\end{proof}

\appendix

\section{Exponential and moment inequalities for dependent sums}\label{Appendix_Exponential inequalities}

\begin{proposition}\label{IntegrabilityForDependentSums}
Assume the real valued random field $Z$ to satisfy Condition \ref{regCond0}.
\begin{enumerate}

\item If Condition~\ref{regCond0}~\ref{DensityRequirement} is satisfied with $a=1$, then for all $j\in\Z$ and $\gamma\in\Z^N$
$$
	\left(\sum_{\gamma\in\Z^N} \E{ |\hat{\theta}_{j,\gamma} - \theta_{j,\gamma} |^2 } \right)^{1/2}\le C_1 L^{d/2} 2^{d/2} \norm{\Phi}_\infty \left\{   \norm{h^{1/2}}_1^{1/2} +  \norm{\tilde{h} }_2 \right\} \frac{|M|^{j/2}}{|I_n|^{1/2} }.
$$

\item If Condition~\ref{regCond0}~\ref{DensityRequirement} is satisfied with $a=2$, then for all $j\in\Z$ and $\gamma\in\Z^N$
$$
	\sum_{\gamma\in\Z^N} \E{ |\hat{\theta}_{j,\gamma} - \theta_{j,\gamma} |^2 }^{1/2} \le C_2 L^{d/2} 2^d \norm{\Phi}_\infty \left\{ \norm{h^{1/4}}_1 +  \norm{\tilde{h} }_1 \right\} \frac{|M|^{j}}{|I_n|^{1/2} }.
$$

\item If Condition~\ref{regCond0}~\ref{DensityRequirement} is satisfied with $a=4$, then for all $j\in\Z$ and $\gamma\in\Z^N$
$$
	\sum_{\gamma\in\Z^N} \E{ |\hat{\theta}_{j,\gamma} - \theta_{j,\gamma} |^2 }^{1/4} \le C_4 L^{d/4} 2^d \norm{\Phi}_\infty^{1/2} \left\{ \norm{h^{1/8}}_1 +  \norm{\tilde{h}^{1/2} }_1 \right\} \frac{|M|^{j}}{|I_n|^{1/4} }.
$$

\end{enumerate}
In all cases the constants $C_1,C_2,C_4\in \R_+$ do not depend on $n\in\N_+^N$. They depend on the bound of the mixing coefficients determined by the numbers $c_0$ and $c_1$ and on the data dimension $d$.

Moreover, the same result is true for $\E{ |\hat{\upsilon}_{k,j,\gamma} - \upsilon_{k,j,\gamma} |^2 } $ for all $k=1,\ldots,|M|$, $j\in\Z$ and $\gamma\in\Z^N$
\end{proposition}

\begin{proof}[Proof of Proposition~\ref{IntegrabilityForDependentSums}]
We only prove the statement concerning the coefficients $\theta_{j,\gamma}$ and assume w.l.o.g. that $j>0$. We begin with the decomposition of the variance 
\begin{align}\begin{split}\label{IntegrabilityForDependentSumsEq1}
		\E{ |\hat{\theta}_{j,\gamma} - \theta_{j,\gamma} |^2 } &\le |I_n|^{-2} \sum_{s\in I_n} \Phi_{j,\gamma}^2 (Z(s)) + |I_n|^{-2} \sum_{ \substack{s,t\in I_n, \\ \norm{s-t} > |M|^{j/N} } } \operatorname{Cov}(\Phi_{j,\gamma}(Z(s)),\Phi_{j,\gamma}(Z(t)) ) \\
		&\quad + |I_n|^{-2} \sum_{ \substack{s,t\in I_n, \\ \norm{s-t} \le |M|^{j/N} } } \operatorname{Cov}(\Phi_{j,\gamma}(Z(s)),\Phi_{j,\gamma}(Z(t)) )
	\end{split}	\end{align}
We easily find that the first summand in \eqref{IntegrabilityForDependentSumsEq1} is at most
\begin{align}\label{IntegrabilityForDependentSumsEq2}
		\frac{ \norm{\Phi}_{\infty}^2 }{ |I_n| } \int_{\R^d} \1{\supp \Phi(\cdot-\gamma)} f(M^{-j}y) \intd{y} \le \frac{ \norm{\Phi}_{\infty}^2 }{ |I_n| } \int_{\R^d} \1{\supp \Phi(\cdot-\gamma)} h(M^{-j}y) \intd{y} .
\end{align}
Consider the second summand, here we apply the inequality of \cite{davydov1968convergence} to bound the covariance by the fourth moments times the mixing coefficient. We obtain the upper bound
\begin{align}
			|&I_n|^{-2}  \sum_{ \substack{s,t\in I_n, \\ \norm{s-t} > |M|^{j/N} } } \E{ \Phi_{j,\gamma}^4 (Z(s))}^{1/2} \alpha(\norm{s-t})^{1/2} \nonumber \\
			&\le \frac{ |M|^{j/2} \norm{\Phi}_{\infty}^2 }{ |I_n|^2 } \sum_{ \substack{s,t\in I_n, \\ \norm{s-t} > |M|^{j/N} } } \left(	\int_{\R^d} \1{\supp \Phi(\cdot-\gamma)} f(M^{-j}y )	\intd{y} \right)^{1/2} \alpha(\norm{s-t})^{1/2} \nonumber \\
			&\le C \frac{ \norm{\Phi}_{\infty}^2 }{ |I_n| } \left(	\int_{\R^d} \1{\supp \Phi(\cdot-\gamma)} h(M^{-j}y )	\intd{y} \right)^{1/2} \sum_{k> |M|^{j/N} } k^{2N-1} \alpha(k)^{1/2},\label{IntegrabilityForDependentSumsEq3}
\end{align}
where we use in the last inequality that $\sum_{ s,t\in I_n, \norm{s-t} > |M|^{j/N} }  \alpha(\norm{s-t})^{1/2} \le C|I_n| \sum_{k>|M|^{j/N}} k^{N-1} \alpha(k)^{1/2}$ for a constant $C$. The third summand can be bounded with the help of the requirement on the joint densities, by assumption we have for all locations $s,t\in\Z^N$ that
$$
			| f_{Z(s),Z(t)} (z_1,z_2) - f(z_1)f(z_2)| \le \tilde{h}(z_1) \tilde{h}(z_2)
$$
for a non-increasing radial function $\tilde{h}$. Consequently, we obtain
\begin{align}\label{IntegrabilityForDependentSumsEq4}
			&|I_n|^{-2} \sum_{ \substack{s,t\in I_n, \\ \norm{s-t} \le |M|^{j/N} } }  \int_{\R^d\times\R^d} |M|^{j} |\Phi(M^j z_1 - \gamma)| |\Phi(M^j z_2 - \gamma)|\, | f_{Z(s),Z(t)} (z_1,z_2) - f(z_1)f(z_2)| \intd{z_1} \intd{z_2} \nonumber \\
			&\le |I_n|^{-1} \norm{\Phi}_{\infty}^2 \int_{\R^d\times\R^d} \1{ \supp \Phi(\cdot - \gamma) \times \supp \Phi(\cdot - \gamma) } (z_1,z_2) \, \tilde{h}(M^{-j} z_1) \tilde{h}(M^{-j} z_2) \intd{z_1}\intd{z_2}.
\end{align}
Note that we have used the relation $\sum_{ {s,t\in I_n, \\ \norm{s-t} \le |M|^{j/N} } } |M|^{-j} \le |I_n|$ in the derivation of the last inequality.

It remains to bound the sum of the variances $\E{ |\hat{\theta}_{j,\gamma} - \theta_{j,\gamma} |^2 }$, respectively the sum of the corresponding standard deviations, respectively the sum of the square root of the standard deviations.. We use the following concept on the integrals from \eqref{IntegrabilityForDependentSumsEq2}, \eqref{IntegrabilityForDependentSumsEq3} and \eqref{IntegrabilityForDependentSumsEq4}: the support of $\Phi(\cdot-\gamma)$ is the cube $[\gamma,\gamma+L e_N]$. Let $y^*_{\gamma}$ be among the points $y$ in this cube such that $M^{-j} y$ is nearest to the origin, i.e., $y^*_{\gamma}$ satisfies $\norm{M^{-j} y^*_{\gamma} }_{\infty} = \inf \left\{ \norm{ M^{-j} y }_{\infty} : y \in [\gamma,\gamma+L e_N] \right \}$. Then we have with the properties of the non-increasing radial functions $h$ and $\tilde{h}$ and for $a=1,2,4,8$
\begin{align*}
		& \sum_{\gamma\in\Z^d} \left(	\int_{\R^d}  \, \1{\supp \Phi(\darg - \gamma )} \,h(M^{-j} z) \intd{z} \right)^{1/a} \le L^{d/a} \sum_{\gamma \in \Z^d} h(M^{-j} y^*_{\gamma} )^{1/a} \le C \, L^{d/a} 2^d  \, \norm{h^{1/a} }_1 \, |M|^{j},
\end{align*}
where the constant $C$ depends on the data dimension $d$. The factor $|M|^j$ in the last inequality is due to a change of variables. Similarly, we obtain for the integral from \eqref{IntegrabilityForDependentSumsEq4} 
$$
\sum_{\gamma\in\Z^d} \left(\int_{\R^d} \1{ \supp \Phi(\cdot - \gamma) }(z)\, \tilde{h}(M^{-j} z) \right)^{2/a} \le C L^{2d/a} 2^d \norm{ \tilde{h}^{2/a} }_1 |M|^j.
$$
This finishes the proof.
\end{proof}

\begin{lemma}\label{LemmaExpInequalityWavelets}
Assume the real valued random field $Z$ to satisfy Condition \ref{regCond0}. Set $q_n \coloneqq B\ln |I_n|$ for some $B>0$. Then there are positive constants $A_1,A_2,A_3$ such that for all $j\in\Z$ and $\gamma\in\Z^N$
\begin{align*}
	\p( |\hat{\theta}_{j,\gamma} - \theta_{j,\gamma} | \ge x ) \le 
	\begin{cases}
			2 \exp\left( - \frac{|I_n| x^2}{A_2 + 2^N A_1 q_n^N |M|^{j/2} x / 3} \right) + A_3 \frac{2^N |M|^{j/2}}{x |I_n|^{c_1 B} },& \quad  \text{for } x \le A_1 |M|^{j/2} \\
			2 \exp\left( - \frac{|I_n| x^2}{A_2 + 2^N A_1 q_n^N |M|^{j/2} x / 3} \right), &  \quad \text{for } x > A_1 |M|^{j/2}.
			\end{cases}
\end{align*}
Here the constant $c_1>0$ is due to the bound on the mixing coefficients and guaranteed by Condition~\ref{regCond0}.

The same result is also true for $\p( |\hat{\nu}_{k,j,\gamma} - \nu_{k,j,\gamma} | \ge x )$ for all $k=1,\ldots,|M|$, $j\in\Z$ and $\gamma\in\Z^N$.
\end{lemma}
\begin{proof}
We use Lemma 4.6 from \cite{LiWavelets}, we only have to replace the factor $2^{j/2}$ by $|M|^{j/2}$. We use that for an rectangular set $\tilde{I}\subseteq \Z^N$ and in case that Condition \ref{regCond0} is satisfied, it is true that for all $k,j,\gamma$
$$
		\E{ \left( \sum_{s\in \tilde{I}} \Psi_{k,j,\gamma}(Z_s) - \E{\Psi_{k,j,\gamma}(Z_s)} \right)^2 } \le C \norm{ \Psi_{k,j,\gamma} }_{\infty}^2 L^{2d} \left( \norm{h}_\infty + \norm{h}_\infty^{1/2} + \norm{\tilde{h} }^2_\infty  \right) |\tilde{I}|,
$$
where the constant $C$ only depends on the lattice dimension $N$ and the mixing coefficients.

Moreover, we bound the probability $P_2(x)$ from Equations (A.4) and (A.7) in this article with the maximum norm of the functions, which is $|M|^{j/2}$ times a constant. This yields the result.
\end{proof}

\begin{lemma}\label{LemmaMomentInequalityWavelets}
Let $q \ge 2$ and $B>0$, set $q_n \coloneqq B\ln |I_n|$. Then it is true that
$$
		\E{ |\hat{\theta}_{j,\gamma} - \theta_{j,\gamma}|^q } \le C_q \left\{ |I_n|^{-q/2} + \left(\frac{|M|^{j/2} q_n^N}{|I_n|}\right)^q + \frac{ |M|^{jq/2} }{|I_n|^{c_1B}}  \right\},
$$
where the constant $C_q$ depends on $B$. The same relation is true for $\E{ |\hat{\upsilon}_{k,j,\gamma} - \upsilon_{k,j,\gamma}|^q }$.
\end{lemma}
\begin{proof}
We use Lemma~\ref{LemmaExpInequalityWavelets}. Define $x^* = \frac{3b_2}{2^N b_1} (q_n^N |M|^{j/2})^{-1}$. Then
\begin{align}
		\E{ |\hat{\theta}_{j,\gamma} - \theta_{j,\gamma}|^q } &= q \int_{0}^\infty x^{q-1} \p(|\hat{\theta}_{j,\gamma} - \theta_{j,\gamma}| \ge x ) \intd{x} \nonumber \\
		& \le 2q \int_{0}^\infty x^{q-1} \left\{  \exp\left( - \frac{|I_n| x^2}{A_2 + 2^N A_1 q_n^N |M|^{j/2} x / 3} \right) + \1{ x \le A_1 |M|^{j/2}} A_3 \frac{|M|^{j/2}}{x |I_n|^{c_1 B} } \right\} \intd{x} \nonumber \\
		\begin{split}\label{LemmaMomentInequalityWaveletsEq1}
		&\le 2q \int_{0}^{x^*} x^{q-1}   \exp\left( - \frac{|I_n| x^2}{2A_2} \right) \intd{x} + 2q \int_{x^*}^{\infty} x^{q-1}  \exp\left( - \frac{|I_n| x^2}{2^N A_1 q_n^N |M|^{j/2} x } \right) \intd{x} \\
		&\quad + 2q A_3\int_0^{A_1 |M|^{j/2}} x^{q-1} \frac{2^N |M|^{j/2}}{x |I_n|^{c_1 B} }  \intd{x}.
		\end{split}
\end{align}
The first integral in \eqref{LemmaMomentInequalityWaveletsEq1} is bounded by
$$
			q \left(\frac{2A_2}{|I_n|}\right)^{q/2} \gamma\left( \frac{q}{2}, C \frac{|I_n|}{q_n^N |M|^{j/2}}  \right) \le C |I_n|^{-q/2},
$$
here we (temporarily) denote the lower incomplete gamma function by $\gamma(\cdot,\cdot)$. Likewise, the second integral is bounded by (modulo a constant)
$$
		\left( \frac{q_n^N |M|^{j/2} }{|I_n|} \right)^q \Gamma\left(q	, C \frac{ |I_n|}{ q_n^{2N} |M|^{j} } \right) \le C \left( \frac{q_n^N |M|^{j/2} }{|I_n|} \right)^q .
$$
Here $\Gamma(\cdot,\cdot)$ is the upper incomplete gamma function. The last integral in \eqref{LemmaMomentInequalityWaveletsEq1} is at most $|M|^{jq/2} |I_n|^{- c_1 B}$ (times a constant). This finishes the proof.
\end{proof}

\section{The question of normalization}\label{Appendix_QuestionOfNormalization}
This appendix contains a result on the convergence of the normalized density estimator: let $p\ge 1$ and $(f_k:k\in\N_+)$ be a sequence of density projections onto (increasing) subspaces of $ L^p(\lambda^d)\cap L^2(\lambda^d)$. Furthermore, let $(\tilde{f}_k: k\in\N_+)\subseteq L^p(\lambda^d \otimes \p) \cap L^2(\lambda^d \otimes \p)$ be a corresponding sequence of density estimators. Define the normalized nonparametric density estimator by
\begin{align}\label{nonparametricDensityEstimator1}
		\hat{f}_k \coloneqq  \frac{1}{S_k} \tilde{f}_k^+, \text{ where } S_k \coloneqq \int_{\R^d} \tilde{f}_k^+ \intd{\lambda}^d
\end{align}
is the normalizing constant.  We have in this case the general result

\begin{proposition}[$L^p$-convergence of $\hat{f}_k$]\label{densityConvergenceLp}
Let $p\in [1,\infty)$ and $f\in L^p(\lambda^d)$ be a density. If the estimator $\tilde{f}_k$ converges to $f$ in $L^p(\lambda^d)$ a.s. and in $L^1(\lambda^d)$ a.s., then $\hat{f}_k$ converges to $f$ in $L^p(\lambda^d)$ a.s. Furthermore, let $\tilde{f}_k$ converge to $f$ in $L^p(\lambda^d\otimes \p)$ and in $L^1(\lambda^d\otimes \p)$; additionally, if $p>1$, let $\liminf_{k \rightarrow \infty} \norm{S_k}_{L^{\infty}(\p) } \ge \delta > 0$. Then the estimator $\hat{f}_k$ converges to $f$ in $L^p(\lambda^d\otimes \p)$.
\end{proposition}
\begin{proof}[Proof of Proposition~\ref{densityConvergenceLp}]
It remains to prove the desired convergence for the term $|\hat{f}_k - \tilde{f}_k|^p$:
\begin{align}\label{densityConvergenceLp2-1}
		&\int_{\R^d} | \hat{f}_k - \tilde{f}_k |^p \,\intd{\lambda}^d \le 2^p \int_{\R^d} (\tilde{f}_k^-)^p \,\intd{\lambda}^d +  2^p \left| 1 - \frac{1}{S_k} \right|^p \,\int_{\R^d} ( \tilde{f}_k^+ )^p \intd{\lambda}^d. 
		\end{align}
Consider the first term in \eqref{densityConvergenceLp2-1},
\begin{align}\label{densityConvergenceLp20}
	\int_{\R^d} |\tilde{f}_k^-|^p \intd{\lambda^d} \le 2^p \int_{\R^d} |f-\tilde{f}_k|^p \, \intd{\lambda^d} + 2^p \int_{\R^d} f^p \I\{ f < f-\tilde{f}_k \} \, \intd{\lambda^d}.
	\end{align}
An application of Lebesgue's dominated convergence theorem shows that the second error in \eqref{densityConvergenceLp20} converges to zero both in the mean and $a.s.$: indeed, we define for $1>\epsilon_1,\epsilon_2 > 0$
\begin{align*}
	L(\epsilon_1) \coloneqq \inf \left\{a\in\R_+: \int_{[-a,a]^d} f^p \, \intd{\lambda}^d \ge 1-\epsilon_1 \right\} < \infty, \quad K(\epsilon_1) \coloneqq [-L(\epsilon_1), L(\epsilon_1)]^d \text{ and } A(\epsilon_2) \coloneqq \{ f > \epsilon_2 \}.
\end{align*}
We get
\begin{align*}
\int_{\{ f < f - \tilde{f}_k\}} f^p \, \intd{\lambda}^d 	&\le \epsilon_1 + \int_{K(\epsilon_1) } f^p \, \1{ f < f - \tilde{f}_k } \intd{\lambda}^d  \\
& \le \epsilon_1 + \int_{K(\epsilon_1) \cap A(\epsilon_2) } f^p\, \1{ \epsilon_2 < |f - \tilde{f}_k| } \intd{\lambda}^d + \epsilon_2^p\lambda^d(K(\epsilon_1)).
	\end{align*}
If $|f-\tilde{f}_k| \rightarrow 0$ in $L^1(\lambda^d\otimes\p)$ and $f\in L^p(\lambda^d)$, then 
\[
	\limsup_{k\rightarrow\infty} \E{ \int_{K(\epsilon_1)\cap A(\epsilon_2)} f^p\, \1{ \epsilon_2 < |f - \tilde{f}_k| } \intd{\lambda}^d } = 0
	\]
with Lebesgue's dominated convergence theorem applied to the measure $\lambda^d\otimes\p$. In the same way, if $|f-\tilde{f}_k| \rightarrow 0$ in $L^1(\lambda^d)$ on a set $\Omega_0 \in \cA$ with $\p(\Omega_0)=1$ and $f\in L^p(\lambda^d)$, then $\limsup_{k\rightarrow\infty} \int_{K(\epsilon_1)\cap A(\epsilon_2)} f^p\, \1{ \epsilon_2 < |f - \tilde{f}_k| } \intd{\lambda}^d = 0$ with Lebesgue's dominated convergence theorem applied to $\lambda^d$ for each $\omega\in\Omega_0$. In addition, this implies $S_k\rightarrow 1$ in the mean and $a.s.$ This finishes the computations on the first term in \eqref{densityConvergenceLp2-1}. We can bound the second term in \eqref{densityConvergenceLp2-1} as
\begin{align}
		 \left| 1 - \frac{1}{S_k} \right|^p \,\int_{\R^d} ( \tilde{f}_k^+ )^p \, \intd{\lambda}^d \le 2^p \left| 1 - \frac{1}{S_k} \right|^p  \int_{\R^d} f^p \, \intd{\lambda^d} + 2^p \left| 1 - \frac{1}{S_k} \right|^p  \int_{\R^d} |\tilde{f}_k - f|^p \, \intd{\lambda^d}. \label{densityConvergenceLp2}
\end{align}
The error $| 1 - 1/S_k|$ on the RHS of \eqref{densityConvergenceLp2} converges to zero $a.s.$ by the continuous mapping theorem. In particular, the RHS of \eqref{densityConvergenceLp2} converges to zero $a.s.$ We come to the convergence in mean. Again by the continuous mapping theorem, the first term on the RHS of \eqref{densityConvergenceLp2} converges to zero in probability. Furthermore, there is a $k^*\in\N_+$ such that for $k\ge k^*$ this term is bounded by $2^p (1+1/\delta)^p \norm{f}_p^p$. Hence, the family $\{ |1-1/S_k|^p: k\ge k^* \}$ is uniformly integrable and this factor converges to zero in the mean. In addition, the first factor in the second term on the RHS of \eqref{densityConvergenceLp2} is bounded for all $k\ge k^*$ and, thus, the whole term converges to zero in the mean.
\end{proof}

\section*{References}


\end{document}